\documentclass[10pt]{siamltex}
\usepackage{amsmath,amsfonts,amscd,amssymb,bm,cite}
\usepackage{mathrsfs,listings,url}
\usepackage{graphics,color,ulem}
\usepackage{float}
\usepackage[titletoc,page]{appendix}
\usepackage{subfigure}
\usepackage[colorlinks=true, bookmarksopen,
pdfauthor={CST},
pdfcreator={pdftex},
pdfsubject={algorithms},
linkcolor={blue},
anchorcolor={black},
citecolor={firebrick},
filecolor={magenta},
menucolor={black},
pagecolor={red},
plainpages=false,pdfpagelabels,
urlcolor={db} ]{hyperref}
\usepackage[capitalise]{cleveref}
\usepackage{comment}
\usepackage{verbatim}
\usepackage{marginnote}
\usepackage{float}
\usepackage[makeroom]{cancel}
\usepackage[pdftex]{graphicx}
\usepackage[section]{placeins}
\usepackage{mathptmx}
\usepackage{cases}
\usepackage{subeqnarray}
\usepackage{hhline}
\usepackage[linesnumbered,ruled,vlined]{algorithm2e}
\usepackage{xcolor}
\usepackage{tikz}
\usepackage{caption} 
\usepackage{lmodern}

\usetikzlibrary{arrows}
\usetikzlibrary{positioning,shapes,snakes,calc,decorations,decorations.markings}

\newtheorem{remark}{Remark}
\newtheorem{scheme}{Scheme}
\newtheorem{assumption}[theorem]{Assumption}

\restylefloat{table}
\allowdisplaybreaks
\definecolor{db}{rgb}{0.0470,0,0.5294}
\definecolor{dg}{rgb}{0.0,0.392,0.0}
\definecolor{firebrick}{rgb}{0.698,0.133,0.133}
\definecolor{bl}{rgb}{0.0,0.0,0.0}
\definecolor{linen}{rgb}{0.980,0.941,0.902}
\definecolor{ivory}{rgb}{1.0,1.0,0.941}
\definecolor{aliceblue}{rgb}{0.941,0.973,1.0}
\definecolor{beige}{rgb}{0.961,0.961,0.863}
\definecolor{tan}{rgb}{0.824,0.706,0.549}
\definecolor{lightsteelblue}{rgb}{0.690,0.769,0.871}
\definecolor{paleturquoise}{rgb}{0.686,0.933,0.933}
\definecolor{lightblue}{rgb}{0.678,0.847,0.902}
\definecolor{skyblue}{rgb}{0.529,0.808,0.922}
\definecolor{palegoldenrod}{rgb}{0.933,0.910,0.667}
\definecolor{lightgoldenrod}{rgb}{0.933,0.867,0.510}
\definecolor{lightyellow}{rgb}{1.0,1.0,0.878}
\definecolor{yellow}{rgb}{1.0,1.0,0.0}
\definecolor{lightyellow1}{rgb}{1.0,1.0,0.878}
\definecolor{lemonchiffon}{rgb}{1.0,0.980,0.804}
\definecolor{myyellow}{rgb}{1,1,.9}
\definecolor{darkgreen}{rgb}{0.0,0.392,0.0}
\definecolor{darkviolet}{rgb}{0.580,0.0,0.827}
\definecolor{lightsalmon}{rgb}{1.0,0.627,0.478}
\definecolor{orange}{rgb}{1.0,0.647,0.0}
\definecolor{darkblue}{rgb}{0.00,0.00,0.55}
\numberwithin{equation}{section}
\Crefname{table}{Table}{Tables}
\Crefname{figure}{Figure}{Figures}

\newcommand\titlelowercase[1]{\texorpdfstring{\lowercase{#1}}{#1}}

\begin{document}
	
	\title{\large{D\titlelowercase{ivergence-free} P\titlelowercase{reserving} M\titlelowercase{ix} F\titlelowercase{inite} E\titlelowercase{lement} M\titlelowercase{ethods} \titlelowercase{for}  F\titlelowercase{ourth-order} A\titlelowercase{ctive} F\titlelowercase{luid} 
    M\titlelowercase{odel}  }} 
	\author{Nan Zheng
			\thanks{
            School of Mathematics, Shandong University, Jinan, Shandong 250100, People’s Republic of China. Email: \href{mailto:202311835@mail.sdu.edu.cn}{202311835@mail.sdu.edu.cn}. } 
            \and  
            Xu Guo
            \thanks{
            Geotechnical and Structural Engineering Research Center, Shandong University, Jinan, 256001, Shandong, China.    
            Email: \href{mailto:guoxu@sdu.edu.cn}{guoxu@sdu.edu.cn}.
            }
            \and  
            Wenlong Pei
            \thanks{
            Department of Mathematics, The Ohio State University, Columbus, OH 43210,USA. Email: \href{mailto:pei.176@osu.edu}{pei.176@osu.edu}.} 
            \and  
            Wenju Zhao
            \thanks{
            School of Mathematics, Shandong University, Jinan, Shandong 250100, People’s Republic of China.
            Email: \href{mailto:zhaowj@sdu.edu.cn}{zhaowj@sdu.edu.cn}.
            }
		    }
	\date{\emty}
	\maketitle

	\begin{abstract}
		This paper is concerned with mixed finite element method (FEM) for solving the two-dimensional, nonlinear fourth-order active fluid equations. 
        By introducing an auxiliary variable $w=-\Delta u$,
        the original fourth problem is transformed into a system of second-order equations, which relaxes the regularity requirements of standard $H^2$-conforming finite spaces. 
        To further enhance the robustness and efficiency of the algorithm, an additional auxiliary variable $\phi$, treated analogously to the pressure, is introduced, leading to a divergence-free preserving mixed finite element scheme. 
        A fully discrete scheme is then constructed by coupling the spatial mixed FEM with the variable-step Dahlquist-Liniger-Nevanlinna (DLN) time integrator. 
        The boundedness of the scheme and corresponding error estimates can be rigorously proven under appropriate assumptions due to unconditional non-linear stability and second-order accuracy of the DLN method. 
        To enhance computational efficiency in practice, we develop an adaptive time-stepping strategy based on a minimum-dissipation criterion. 
        Several numerical experiments are displayed to fully validate the theoretical results and demonstrate the accuracy and efficiency of the scheme for complex active fluid simulations.

	\end{abstract}
	
	

    \section{Introduction}
    In this paper, we consider a divergence-free persevered mixed finite element method for solving the following non-linear fourth-order active fluid equations on the domain 
    $D \subset \mathbb{R}^2$ and time interval $[0,T]$: the velocity of fluid $u(x,t)$ and pressure $p(x,t)$ are governed by 
    \begin{equation}\label{ACEs}
        \begin{cases}
            u_{t} - \mu \Delta u + \gamma \Delta^2 u+\nu (u \cdot \nabla)u  + \rho u +  \lambda |u|^2 u+ \nabla p = f,    
            &\text{in}  ~D \times (0, T], \\
            \nabla \cdot u = 0,   
            &\text{in} ~ D \times (0, T],
        \end{cases}
    \end{equation}
    subject to the following initial-boundary conditions: 
    \begin{equation*}
        u(x,0) = u_{0}~ \text{in} ~D, \indent
        u =\Delta u =  0 ~\text{on}~ \partial D.
    \end{equation*}
    Here, the domain $D$ is bounded with Lipschitz boundary 
    $\partial D$, and $f$ is the source function. 
    Non-negative parameters $\mu, \gamma, \nu$ represent the viscosity coefficient, the generic stability coefficient, and the density coefficient, respectively.
    Terms $\rho u$ ($\rho \in \mathbb{R}$) and $\lambda |u|^2 u$ ($\lambda \geq 0$) correspond to a quartic Landau velocity potential 
    \cite{ramaswamy2019active,wensink2012meso,toner1998flocks}.

    Active fluid, consisting of self-propelled particles capable of converting energy into motion, represents a distinctive class of nonequilibrium systems \cite{ramaswamy2019active,annurev-fluid-010816-060049}. 
    In recent years, the mathematical modeling and numerical analysis of active fluid have increasingly attracted attention due to their significant applications \cite{fielding2022rheology,MR4049919,MR4131825,MR4736040}, since these types of systems can capture the complex fluid behaviors exhibited by active suspensions, such as those found in bacterial superfluids \cite{pnas.1722505115}, turbulence in microswimmer suspensions  \cite{qi2022emergence} and  similar active suspensions.

    The mathematical characterization of active fluid dynamics often involves generalized forms of the Navier-Stokes equations augmented with higher-order dissipative terms and nonlinear active forcing. 
    Despite the effectiveness in capturing rich dynamical behaviors, the non-linear system of fourth-order equations \eqref{ACEs} presents substantial analytical and numerical challenges.
    To address these issues and reduce the regularity requirements, we reformulate the original fourth-order active fluid equations into a system of second-order equations. 
    By introducing an auxiliary variable  \( w = -\Delta u \),
    we achieve the following equivalent reformulation of \eqref{ACEs}
    \begin{equation}\label{ACEs 2nd}
        \begin{cases}
            u_{t} - \mu \Delta u - \gamma \Delta w + \nu (u \cdot \nabla) u + \rho u + \lambda |u|^2 u + \nabla p = f, &\text{in} ~ D \times (0, T], \\
            w = -\Delta u, &\text{in} ~ D \times (0, T], \\
            \nabla \cdot u = 0, &\text{in} ~ D \times (0, T],
        \end{cases}
    \end{equation}
    with the modified initial-boundary conditions:
    \begin{equation} \label{ACEs-2nd-2}
        u(x,0) = u_{0} ~ \text{in} ~ D, \quad u = w = 0 ~ \text{on} ~ \partial D.
    \end{equation}
    which effectively reduces the complexity associated with the biharmonic operator $\Delta^2 u$.
    The resulting system \eqref{ACEs 2nd}-\eqref{ACEs-2nd-2} is eligible for the use of finite element methods based on $H_0^1$-conforming basis functions without 
    the restrictive \( H^2 \)-regularity requirements.
    More importantly, the auxiliary variable inherently satisfies a divergence-free constraint 
    $(\nabla \cdot w = 0)$ preserving physical fidelity and the incompressibility condition, which are crucial for realistic simulations \cite{da2025error,da2017divergence}.

    Finite element methods (FEM), known for their robustness, flexibility, and effectiveness, are extensively employed in spatial discretizations of Navier-Stokes system related fluid dynamics models \cite{MR4736040,MR4293957, he2007stability,MR4835947,abgrall2023hybrid,abgrall2020analysis,ayuso2005postprocessed,MR3147678,lan2025robust}.
    For temporal discretization, a variety of approaches have been thoroughly analyzed, including the Euler scheme, Crank-Nicolson related scheme, Runge-Kutta scheme and backward differentiation formula scheme \cite{ern2022invariant,MR4835947,AAMM-16-5,hou2025unconditionally,baker2024numerical,banjai2012runge,BAI20123265,decaria2022general,gottlieb2022high}, etc. 
    Recently, much effort has been devoted to the numerical analysis of variable time-stepping schemes and their potential for time adaptivity \cite{ait2023time,dahlquist1983stability,MR4328513,MR4471049}. 
    The family of Dahlquist-Liniger-Nevanlinna (DLN) methods (with one parameter $\theta \in [0,1]$), which ensures unconditional stability and second-order accuracy for general dissipative nonlinear systems with arbitrary time grids \cite{dahlquist1983stability,LPT21_AML,LPT23_ACSE},
    has been proven successful in simulations of stiff differential equations and complicated fluid models \cite{QHPL21_JCAM,layton2022analysis,MR4866010,pei2024semi,CLPX2025_JSC,SP24_IJNAM}. 
    Given the time interval $[0,T]$ with its partition $\{ t_{n} \}_{n=0}^{M}$ and 
    the initial value problem $y'(t) = g(t,y(t))$ with $t \in [0,T], \ y(0) = y_{0} \in \mathbb{R}^{d}$, 
    the family of variable time-stepping DLN methods for the problem reads 
    \begin{gather}
		\sum_{\ell =0}^{2}{\alpha _{\ell }}y_{n-1+\ell }
		= \widehat{k}_{n} g \Big( \sum_{\ell =0}^{2}{\beta _{\ell }^{(n)}}t_{n-1+\ell } ,
		\sum_{\ell =0}^{2}{\beta _{\ell }^{(n)}}y_{n-1+\ell} \Big), \qquad n=1,\ldots,M-1.
		\label{eq:1leg-DLN}
	\end{gather}
    $y_{n}$ represents the DLN solution of $y(t)$ at time $t_{n}$. The coefficients of the DLN method in \eqref{eq:1leg-DLN} are 
	\begin{gather}
        \label{DLNcoeff}
		\begin{pmatrix}
			\alpha _{2} \vspace{0.2cm} \\
			\alpha _{1} \vspace{0.2cm} \\
			\alpha _{0} 
		\end{pmatrix}
		= 
		\begin{pmatrix}
			\frac{1}{2}(\theta +1) \vspace{0.2cm} \\
			-\theta \vspace{0.2cm} \\
			\frac{1}{2}(\theta -1)
		\end{pmatrix}, \ \ \ 
		\begin{pmatrix}
			\beta _{2}^{(n)}  \vspace{0.2cm} \\
			\beta _{1}^{(n)}  \vspace{0.2cm} \\
			\beta _{0}^{(n)}
		\end{pmatrix}
		= 
		\begin{pmatrix}
			\frac{1}{4}\Big(1+\frac{1-{\theta }^{2}}{(1+{%
					\varepsilon _{n}}{\theta })^{2}}+{\varepsilon _{n}}^{2}\frac{\theta (1-{%
					\theta }^{2})}{(1+{\varepsilon _{n}}{\theta })^{2}}+\theta \Big)\vspace{0.2cm%
			} \\
			\frac{1}{2}\Big(1-\frac{1-{\theta }^{2}}{(1+{\varepsilon _{n}}{%
					\theta })^{2}}\Big)\vspace{0.2cm} \\
			\frac{1}{4}\Big(1+\frac{1-{\theta }^{2}}{(1+{%
					\varepsilon _{n}}{\theta })^{2}}-{\varepsilon _{n}}^{2}\frac{\theta (1-{%
					\theta }^{2})}{(1+{\varepsilon _{n}}{\theta })^{2}}-\theta \Big)%
		\end{pmatrix}.
	\end{gather}
    $\varepsilon _{n} = (k_n - k_{n-1})/(k_n + k_{n-1}) \in (-1,1)$ is the step variability.
    The method is reduce to the midpoint rule on $[t_{n}, t_{n+1}]$ if $\theta = 1$ and midpoint rule on $[t_{n-1}, t_{n+1}]$ if $\theta = 0$.

    Motivated by these advances, our study integrates mixed finite element spatial discretization with DLN temporal discretization, and thus proposes a computationally efficient, stable, and accurate framework for analyzing the sophisticated dynamics exhibited by active fluid.
    In addition, an adaptive time-stepping strategy, inspired by the minimal dissipation criterion of Capuano et al. \cite{capuano2017minimum}, is further developed to balance the conflict between accuracy and computational costs.

    The main contributions of this report are to:  
    \begin{itemize}
        \item construct a divergence-free preserving mixed finite element spatial discretization for the fourth-order active fluid equations \eqref{ACEs}, which reduces complexity and relaxes regularity requirements,
        \item     
        utilize the variable time-stepping DLN temporal integrator for full discretization and present rigorous proof that the fully discrete algorithm is unconditionally stable in kinetic energy under arbitrary time grids,
        \item 
        carry out a detailed error estimate for both velocity and pressure under moderately relaxed regularity requirements and time step constraints,
        \item
        design a time adaptive strategy by the minimal dissipation criterion, which  
        significantly enhance time efficiency in practice.
    \end{itemize}

    The remainder of this paper is structured as follows.
    In Section \ref{sec:sec2}, we introduce necessary preliminaries and notations. 
    In Section \ref{sec:sec3}, we establish the fully discrete scheme for the fourth-order active fluid equations \eqref{ACEs} based on a divergence-free preserving mixed finite element spatial discretization and variable time-stepping DLN time integrator. 
    By adding a mild time step restriction, 
    we prove rigorously that the scheme is long-time stable in kinetic energy under an arbitrary time step sequence. 
    Detailed error estimates of the resulting fully discrete scheme are presented 
    in Section \ref{sec:sec4}.
    In Section \ref{sec:sec5}, a series of numerical experiments are offered to validate the theoretical findings.
    Section \ref{sec:sec6} summarizes the main results of the report and outlines potential directions for future research.

    \section{Notation and Preliminaries}
    \label{sec:sec2}
    For \( 1 \leq p \leq \infty \) and \( r \in \mathbb{N}^+ \cup \{ 0 \} \), 
    \( W^{r,p}(D) \) represents the usual Sobolev space with norm \( \|\cdot\|_{W^{r,p}} \). 
    The case $p=2$ is reduced to the Hilbert space $H^r(D)$ with norm \( \|\cdot\|_r \).
    In particular, the Lebesgue space \( L^2(D) := H^{0}(D) \) is endowed with the standard inner product $(\cdot, \cdot)$  and norm \( \|\cdot\| \) (or $\| \cdot \|_{0}$). 
    We need the following Bochner spaces
    \begin{align*}
        L^{\infty}([0,T];H^r(D)) &:= \big\{ v(x,t): 
        \| v \|_{\infty, r} = \sup_{0 \leq t \leq T} \|v(t)\|_r < \infty \big\}, \\
        L^p([0,T]; H^r(D)) &:= \big\{ v(x,t): 
        \| v \|_{p,r} = \Big( \int_{0}^{T}\|v(t)\|_r^p\mathrm{d}t \Big)^\frac{1}{p} 
        < \infty \big\}.
    \end{align*}
    The solution spaces for velocity $u$ and pressure $p$ in \eqref{ACEs} are 
    \begin{align*}
        X = \Big\{ u \in [H^1(D)]^2 : u = 0 ~\text{on}~ \partial D \Big\},
        \quad 
        &
        Q = \Big\{ p \in L^2(D) : \int_{D} p \,\mathrm{d}x = 0 \Big\}.
    \end{align*}
    The divergence-free space $V$ is 
    \begin{equation}\label{V_space}
        V = \big\{ v \in X: \nabla \cdot v =0 ~\text{in} ~D  \big\}.
    \end{equation}
    We define the skew-symmetric trilinear form
    \begin{align}\label{b form}
        b(u,v,w) = \frac{1}{2}(u \cdot \nabla v, w) - \frac{1}{2}(u \cdot \nabla w, v), \quad 
        u,v,w \in X.
    \end{align}
    We have the following estimates for $b(\cdot,\cdot,\cdot)$ (see \cite{he2007stability,MR1043610,temam2001navier,Ing13_IJNAM} for proof)
    \begin{subequations} 
        \begin{align}
            &b(u,v,w) \leq C \| \nabla u \| \|\nabla v\| \|\nabla w\|,  \quad \forall u,v,w \in X,\label{b 1 1 1}
            \\
            &b(u,v,w) \leq C \| u \|^{\frac{1}{2}} \|\nabla u\|^{\frac{1}{2}} \|\nabla v\| \|\nabla w\|,  \quad \forall u,v,w \in X, \label{b 0 1 1 1}
            \\
            &b(u,v,w) \leq C \| u \|_{1} \| v \|_{1} \| w \|^{\frac{1}{2}} 
            \| w \|_{1}^{\frac{1}{2}},   \label{b 1 1 0 1} \\
            &b(u,v,w) \leq C \|u\| \|v\|_2 \|w\|_1, \quad \forall u,w \in X, \ v \in [H^2(D)]^2, \label{b 0 2 1} \\
            &b(u,v,w) \leq C \|u\|_{2} \|\nabla v \| \|w\|, \quad \forall u,w \in X, \ 
            v \in \big([H^2(D)]^2 \big)  \cap X, \label{b 2 1 0} \\
            &b(u,v,w) \leq C \|u\|_{1} \| v \|_{2} \|w\|, \quad \forall u,w \in X, \ 
            v \in \big([H^2(D)]^2 \big) \cap X, \label{b 1 2 0} 
        \end{align}
    \end{subequations}
    where \( C > 0 \) is a positive constant which only depends on the domain \( D \).
    The variational formulation of \eqref{ACEs 2nd}-\eqref{ACEs-2nd-2} is: the pair 
    $(u,w,p) \in (X,V,Q)$ satisfies
    \begin{align} 
        \begin{split} \label{variational-formula-1} 
            &(u_{t}, v)
            + \mu(\nabla u, \nabla v) 
            +\gamma (\nabla w,\nabla v)
            + \nu b( u,  u, v)  
            + \rho (u,v) 
            + \lambda (|u|^2 u ,v)
            \\
            &\quad 
            - (p, \nabla \cdot v)
            = (f, v), 
        \end{split}
        \\
        &(w,\varphi) = (\nabla u,\nabla \varphi), \label{variational-formula-2}
        \\
        &(\nabla \cdot u, q) = 0.  \label{variational-formula-3}
    \end{align} 
    for all $(v,\varphi,q) \in (X,V,Q)$ and all $t \in (0,T]$. 
    By introducing an auxiliary variable $\phi \in Q$, we have the following equivalent formulation: the pair $(u,w,\phi,p) \in (X,X,Q,Q)$ satisfies 
    \begin{align} 
        \begin{split} \label{variational_formula_second_1} 
            &(u_{t}, v)
            + \mu(\nabla u, \nabla v) 
            +\gamma (\nabla w,\nabla v)
            + \nu b( u,  u, v)  
            + \rho (u,v) 
            + \lambda (|u|^2 u ,v)
            \\
            &\quad 
            - (p, \nabla \cdot v)
            = (f, v), 
        \end{split}
        \\
        &(w,\varphi)-(\phi,\nabla \cdot \varphi) = (\nabla u,\nabla \varphi), \label{variational_formula_second_2}
        \\
        &(\nabla \cdot u, q) = 0,  \label{variational_formula_second_3}
        \\
        &(\nabla \cdot w, \zeta) = 0,  \label{variational_formula_second_4}
    \end{align}
    for all $(v,\varphi,\zeta,q) \in (X,X,Q,Q)$ and all $t \in (0,T]$.
    To handle the non-linear term  $|u|^2 u$ in equations \eqref{variational-formula-1} 
    and  \eqref{variational_formula_second_1}, 
    we need the following lemma about monotonicity and continuity properties on $\mathbb{R}^{2}$ (see \cite{glowinski1975approximation,deugoue2022numerical} for proof)
    \begin{lemma}\label{vector_norm_inequality_lem} 
        For all $x, y \in \mathbb{R}^2$, the following inequalities hold:  
        \begin{align}
            &\text{monotonicity: } 
            \big( |x|^{2} x - |y|^{2} y, x-y \big)_{\mathbb{R}^2}  \geq C |x-y|^4,  \label{vector_norm_inequality_1}  \\
            &\text{continuity: } \ \ \ \ 
            \big| |x|^{2} x - |y|^{2} y \big| \leq C \big( |x| + |y| \big)^{2} |y - x|,  \label{vector_norm_inequality_2}
        \end{align}
        where $C > 0$ is a constant independent of $x$ and $y$, and \( (\cdot,\cdot)_{\mathbb{R}^2} \) denotes the standard Euclidean inner product on $\mathbb{R}^2$.
    \end{lemma}
    Given arbitrary sequence $\{ z_{n} \}_{n=0}^{\infty}$,  we adopt the following notations for convenience
    \begin{align*}
        z_{n,\alpha} = \alpha_{2} z_{n+1} + \alpha_{0} z_{n} + \alpha_{0} z_{n-1}, \qquad
        z_{n,\beta} = \beta _{2}^{(n)} z_{n+1} + \beta _{1}^{(n)} z_{n} 
        + \beta _{0}^{(n)} z_{n-1}, 
    \end{align*}
    where $\{ \alpha_{\ell} \}_{\ell=0}^{2}$ and $\{ \beta_{\ell}^{(n)} \}_{\ell=0}^{2}$ are coefficents of the DLN method in \eqref{DLNcoeff}.

    Let $\mathscr{T}_h \  (0<h<1)$ be a regular partition of $\overline{D} = \cup_{K\in\mathscr{T}_h}\overline{K}$ with the mesh size $h$,
    $P_r(K)$  denotes the space of polynomials of degree less than or equal to $r$ on $K \in\mathscr{T}_h$ for any positive integer $r$. 
    The finite element spaces for $X$, $Q$ and $V$ are 
    \begin{align}
        X_{h} &= \Big\{ v^{h} \in [C^{0}(\bar{\Omega})]^d \cap X : v^{h} \mid_{K} \in [P_{r+1}(K)]^d, \ \forall K \in \mathscr{T}_{h} \Big\},   \label{zeq:FEM-space-Xh-1}\\
        Q_{h} &= \Big\{ q^{h} \in C^{0}(\bar{\Omega}) \cap L^2(D) : q^{h} \mid_{K} \in P_{r}(K), \ \forall K \in \mathscr{T}_{h} \Big\},   \label{zeq:FEM-space-Qh-1} \\
        V_{h} &= \Big\{ v^{h} \in X_{h} : \big( \nabla \cdot v^{h}, q^{h} \big) = 0, \ \forall q^{h} \in Q_{h} \Big\}.  \label{zeq:FEM-space-Vh-1}
    \end{align}
    We assume that the spaces $X_h \times Q_h$ in \eqref{zeq:FEM-space-Xh-1}-\eqref{zeq:FEM-space-Qh-1} satisfy the Ladyzhenskaya-Babuska-Brezzi condition ($LBB^h$ condition):
    for any $ q^h \in Q_h$, there exists $C > 0$ such that 
    \begin{equation}\label{LBB}
        \|q^h\| \leq C \sup_{v^h \in X_h \setminus \{0\}} \frac{(\nabla \cdot v^h, q^h)}{\|\nabla v^h\|}.
    \end{equation}
    Taylor--Hood element space, Mini element space, and {Scott--Vogelius element} space~\cite{MR813691} are all typical finite element spaces having $LBB^h$ condition in \eqref{LBB}. 
    We choose Taylor--Hood element space ($r=1$) for spatial discretization throughout our work.
    $X_h$ satisfies the inverse inequality
    \begin{align}
        \label{inverse-ineq}
        \| \nabla v^h \| \leq C h^{-1} \| v^h \|, \qquad \forall v^h \in X_h.
    \end{align}
    Here $C$ is a positive constant independent of the mesh diameter $h$.

    For any $(u, w, p) \in (X, V, Q)$, we define the Stokes-type projection 
    $\mathcal{S}_h (u, w, p) := (\mathcal{S}_h u, \mathcal{S}_h w, \mathcal{S}_h p) \in (X_h, V_h, Q_h)$ to be the unique pair such that for all $(v^h, \varphi^h, q^h) \in (X_h, V_h, Q_h)$
    \begin{equation}\label{Stokes-type-projection-defination}
        \begin{cases}
            \mu(\nabla(u - \mathcal{S}_{h} u), \nabla v^h) 
            + \gamma (\nabla(w - \mathcal{S}_{h} w), \nabla v^h) 
            - (p - \mathcal{S}_{h} p, \nabla \cdot v^h) = 0, 
            \\
            (w - \mathcal{S}_{h} w, \varphi^h) 
            = (\nabla(u - \mathcal{S}_{h} u), \nabla \varphi^h),
            \\
            (\nabla \cdot \mathcal{S}_h u, q^h) = 0.
        \end{cases}
    \end{equation}
    For the purpose of numerical implementation, we reformulate $\mathcal{S}_h$ in the following equivalent form:
    $\mathcal{S}_h (u, w, \phi, p) \in (X, X, Q, Q) := (\mathcal{S}_h u, \mathcal{S}_h w, \mathcal{S}_h \phi, \mathcal{S}_h p) \in (X_h, X_h, Q_h, Q_h)$ to be the unique pair such that 
    for all $(v^h, \varphi^h, \zeta^h, q^h) \in (X_h, X_h, Q_h, Q_h)$
    \begin{equation}\label{Stokes-type-projection-defination 2}
        \begin{cases}
            \mu(\nabla(u - \mathcal{S}_{h} u), \nabla v^h) 
            + \gamma (\nabla(w - \mathcal{S}_{h} w), \nabla v^h) 
            - (p - \mathcal{S}_{h} p, \nabla \cdot v^h) = 0, 
            \\
            (w - \mathcal{S}_{h} w, \varphi^h) - (\phi - \mathcal{S}_{h} \phi, \nabla \cdot \varphi^h)
            = (\nabla(u - \mathcal{S}_{h} u), \nabla \varphi^h),
            \\
            (\nabla \cdot \mathcal{S}_h u, q^h) = 0,
            \\
            (\nabla \cdot \mathcal{S}_h w, \zeta^h) = 0.
        \end{cases}
    \end{equation}
    The projection $\mathcal{S}_h$ in \eqref{Stokes-type-projection-defination} or \eqref{Stokes-type-projection-defination 2} satisfies the following approximation 
    \begin{equation}\label{Stokes-type projection}
        \begin{split}
            &\| u - \mathcal{S}_h u\| + \| w - \mathcal{S}_h w\|
            + h \big\{
            \|\nabla (u - \mathcal{S}_h u)\| + \|\nabla (w - \mathcal{S}_h w)\| + \|p -  \mathcal{S}_h p\| \big\} 
            \\
            &\leq C h ^{r+2} \big\{
            \| u \|_{r+2} + \| w \|_{r+2} + \| p \|_{r+1}
            \big\}.
        \end{split}
    \end{equation}
    The proof of \eqref{Stokes-type projection} is very similar to that of the approximation of Stokes projection, thus we refer to \cite{GR86_Springer,Joh16_Springer} for details.

    \section{Fully discrete schemes}
    \label{sec:sec3}
    The fully discrete DLN scheme for the non-linear, fourth-order active fluid model 
    in \eqref{ACEs} is 
    \begin{scheme} \label{fully discrete formulations scheme second}
		Given $u_n^h, u_{n-1}^h \in X_h$, $w_n^h, w_{n-1}^h, \in V_h$ and $p_n^h, p_{n-1}^h \in Q_h$, 
		we solve $(u_{n+1}^h, w_{n+1}^h,p_{n+1}^h) \in( X_h, V_h,Q_h)$ such that 
		\begin{align} \label{fully discrete formulations second 1}
			\begin{split}
				&\frac{1}{\widehat{k}_n}(u_{n,\alpha}^h, v^h)
				+ \mu(\nabla u_{n,\beta}^h, \nabla v^h) 
				+\gamma (\nabla w_{n,\beta}^h,\nabla v^h)
				+ \nu b( u_{n,\beta}^h,  u_{n,\beta}^h, v^h) 
				\\
				&\quad  
				+ \rho (u_{n,\beta}^h,v^h) 
				+ \lambda (|u^h_{n,\beta}|^2 u^h_{n,\beta} ,v^h)
				- (p_{n,\beta}^h, \nabla \cdot v^h)
				= (f_{n,\beta}, v^h), 
			\end{split}
			\\
			\begin{split}
				&(w_{n+1}^h,\varphi^h)= (\nabla u_{n+1}^h,\nabla \varphi^h),\label{fully discrete formulations second 2}
			\end{split}
			\\
			\begin{split}
				&(\nabla \cdot u_{n+1}^h, q^h) = 0, \label{fully discrete formulations second 3}
			\end{split}
		\end{align} 
        for all $(v^h, \varphi^h, q^h) \in ( X_h, V_h,Q_h)$ and $n = 1, 2, \cdots, M-1$.
	\end{scheme}

    In the beginning, we set $u_0^h = \mathcal{S}_h u_0$ and employ the fully-implicit Crank-Nicolson scheme to solve for $u_1^h$. 
    Accordingly, numerical solutions at two initial steps are second-order accurate in time.
    As the auxiliary variable $w$ is a divergence-free function, we introduce the following equivalent divergence-free DLN fully discrete scheme for practical use. 
    \begin{scheme}\label{fully discrete formulations scheme}
        Given $u_n^h, u_{n-1}^h, w_n^h, u_{n-1}^h, \in X_h$ and $p_n^h, p_{n-1}^h \in Q_h$, 
		we solve $(u_{n+1}^h, w_{n+1}^h,\phi_{n+1}^h, p_{n+1}^h) \in ( X_h, X_h, Q_h,Q_h)$ such that
        \begin{align} \label{fully discrete formulations 1}
            \begin{split}
                &\frac{1}{\widehat{k}_n}(u_{n,\alpha}^h, v^h)
                + \mu(\nabla u_{n,\beta}^h, \nabla v^h) 
                +\gamma (\nabla w_{n,\beta}^h,\nabla v^h)
                + \nu b( u_{n,\beta}^h,  u_{n,\beta}^h, v^h) 
                \\
                &\quad  
                + \rho (u_{n,\beta}^h,v^h) 
                + \lambda (|u^h_{n,\beta}|^2 u^h_{n,\beta} ,v^h)
                - (p_{n,\beta}^h, \nabla \cdot v^h)
                = (f_{n,\beta}, v^h), 
            \end{split}
            \\
            \begin{split}
                &(w_{n+1}^h,\varphi^h) -(\phi_{n+1}^h,\nabla \cdot \varphi^h)= (\nabla u_{n+1}^h,\nabla \varphi^h),\label{fully discrete formulations 2}
            \end{split}
            \\
            \begin{split}
                &(\nabla \cdot u_{n+1}^h, q^h) = 0, \label{fully discrete formulations 3}
            \end{split}
            \\
            \begin{split}
                &(\nabla \cdot w_{n+1}^h, \zeta^h) = 0. \label{fully discrete formulations 4}
            \end{split}
        \end{align} 
        for all $(v^h, \varphi^h, \zeta^h, q^h) \in ( X_h, X_h,Q_h,Q_h)$ and $n = 1, 2, \cdots, M-1$.
    \end{scheme}

    \subsection{Stability Analysis}
    In this subsection, we prove that Scheme \ref{fully discrete formulations scheme second} 
    or Scheme \ref{fully discrete formulations scheme} preserves the property of unconditional boundedness of kinetic energy due to $G$-stability of property of the variable time-stepping DLN method.
    \begin{definition}
		For $\theta \in [0 , 1]$ and  $u, v \in [L^2(D)]^2$, the $G$-norm $\| \cdot \|_{G(\theta)}$ is defined as
		\begin{equation} \label{G-norm}
			\left\Vert
			\begin{array}{cc}
				u \\
				v
			\end{array}
			\right\Vert^2_{G(\theta )}
			=
			\int_{D}
			\left[u^\top v^\top \right] G(\theta) 
			\left[
			\begin{array}{cc}
				u \\
				v
			\end{array}
			\right]
			\mathrm{d}D
			=\frac{1}{4}(1+\theta)\|u\|^2 +\frac{1}{4}(1-\theta)\|v\|^2,
		\end{equation}
		where the notation $\top$ represents the transpose of a vector and $G(\theta)$ is a symmetric semi-positive definite matrix with $\mathbb{I}_{2\times 2}$ identity matrix defined as
		\[
		G(\theta) = \begin{bmatrix}
			\frac{1}{4}(1+\theta)\mathbb{I}_{2\times 2} & 0 \\
			0 & \frac{1}{4}(1-\theta)\mathbb{I}_{2\times 2}
		\end{bmatrix}.
		\]
	\end{definition}
    \begin{lemma} \label{G-stable-lemma}
        For any sequence $\{ v_{n} \}_{n=0}^{M} \subset [L^2(D)]^2$, the following identity holds
		\begin{equation}\label{G-stable}
			\big( v_{n,\alpha}, v_{n,\beta} \big)  
			=
			\Big\lVert
			\begin{array}{cc}
				v_{n+1} \\
				v_n
			\end{array}
			\Big\lVert^2_{G(\theta )}
			-
			\Big\lVert
			\begin{array}{cc}
				v_{n} \\
				v_{n-1}
			\end{array}
			\Big\lVert^2_{G(\theta )}
			+
			\Big\| \sum_{\ell=0}^{2} a_{\ell}^{(n)} v_{n-1+{\ell}} \Big\|^2,
		\end{equation}
        holds for all $n = 1,2, \cdots, M-1$ and any fixed $\theta \in [0,1]$. Here $\{a_{\ell}^{(n)}\}_{\ell=0}^2$ 
        in \eqref{G-stable} are 
		\begin{equation*} 
			a_1^{(n)}  = -\frac{\sqrt{\theta (1-\theta^2)}}{\sqrt{2}(1+\varepsilon_n \theta)},  \quad
			a_2^{(n)} = -\frac{1-\varepsilon_n}{2} a_1^{(n)} ,   \quad
			a_0^{(n)} = -\frac{1+\varepsilon_n}{2} a_1^{(n)}.   
		\end{equation*}
	\end{lemma}
    \begin{proof}
        The proof of $G$-stability identity in \eqref{G-stable} is just a algebraic calculation. 
    \end{proof}
    \ \\
    \begin{theorem}[Boundedness in kinetic energy] \label{Stability of DLN method}
		Assume that $ u_0^h, u_1^h \in X_h $, $f \in L^2([0,T],[L^2(D)]^2)$ and time steps satisfies
		\begin{align} \label{time-cond-stab}
			C_{\beta}^{(n)} |\rho| \widehat{k}_{n} \leq \frac{1+\theta}{4}, \quad \forall \ n = 1, \cdots M-1, 
		\end{align}
		where $\displaystyle C_{\beta}^{(n)} = \sum_{\ell = 0}^{2} \big( \beta_{\ell}^{(n)} \big)^{2}$.
		Scheme \ref{fully discrete formulations scheme second} is bounded in kinetic energy, i.e. for $ 1 <m \leq M $, 
		\begin{align} 
				&\frac{1+\theta}{4}\| u^h_{m} \|^2 +
				\sum_{n=1}^{m-1} \Big[ \Big\| \sum_{\ell=0}^{2} a_{\ell}^{(n)} u^h_{n-1+ \ell} \Big\|^2 
				+ \widehat{k}_n \Big( \frac{\mu}{2} \| \nabla u^h_{n,\beta}\|^2
				+ \gamma \| w_{n,\beta}^h \|^2 + \lambda \| u^h_{n,\beta} \|_{L^4}^4 \Big) \Big]
				\notag \\
				&\leq \exp \big( C(\theta) T \big) \Big(  \frac{1}{4}(1+\theta) \| u^h_{1} \|^2
				+ \frac{1}{4}(1+\theta) \| u^h_{0} \|^2 
				+ \frac{C}{\mu} \sum_{n=0}^{m-1}  \widehat{k}_{n} \| f_{n,\beta} \|^2 \Big). 
				\label{Stability of DLN inequality}
		\end{align}
	\end{theorem}
    \begin{proof}
        We set $v^h = u^h_{n,\beta}$ in \eqref{fully discrete formulations second 1}, 
        $\varphi^h = w^h_{n,\beta}$ in \eqref{fully discrete formulations second 2}, 
        $q^h = p^h_{n,\beta}$ in \eqref{fully discrete formulations second 3}, 
        and add three equations together to derive 
		\begin{align}
			&\frac{1}{\widehat{k}_n} (u_{n,\alpha}^h, u^h_{n,\beta}) 
			+ \mu \| \nabla u^h_{n,\beta} \|^2 + \gamma \| w_{n,\beta}^h \|^2
			+\rho \| u^h_{n,\beta} \|^2 + \lambda \| u^h_{n,\beta} \|_{L^4}^4 
            + \nu b(u_{n,\beta}^h, u_{n,\beta}^h, u^h_{n,\beta}) \notag \\
			&= (f_{n,\beta}, u^h_{n,\beta}).
            \label{Stab-eq1}
		\end{align}
        By the skew-symmetric property of $b(\cdot,\cdot,\cdot)$, Cauchy--Schwarz inequality and Poincar\'{e}'s inequality, \eqref{Stab-eq1} becomes
        \begin{align*}
			\frac{1}{\widehat{k}_n} (u_{n,\alpha}^h, u^h_{n,\beta}) 
			+ \mu \| \nabla u^h_{n,\beta} \|^2 + \gamma \| w_{n,\beta}^h \|^2
			+\rho \| u^h_{n,\beta} \|^2 + \lambda \| u^h_{n,\beta} \|_{L^4}^4
			\leq C \|f_{n,\beta}\| \|\nabla u^h_{n,\beta}\|.
		\end{align*}
        We apply $G$-stability identity in \eqref{G-stable} and Young's inequality to the above inequality 
        \begin{align} \label{Stab-eq2}
			\begin{split}
				&\bigg\| 
				\begin{array}{cc}
					u^h_{n+1} \\
					u^h_{n} \\
				\end{array} \bigg\|^2_{G(\theta )}
				- \bigg\|
				\begin{array}{cc}
					u^h_{n} \\
					u^h_{n-1} \\
				\end{array} \bigg\|^2_{G(\theta )}
				+ \Big\| \sum_{\ell=0}^{2} a_{\ell}^{(n)} u^h_{n-1+ \ell} \Big\|^2
				+ \frac{\mu}{2} \widehat{k}_n \| \nabla u^h_{n,\beta} \|^2 
				\\
				&+\rho \widehat{k}_n \| u^h_{n,\beta} \|^2
				+ \gamma \widehat{k}_n \| w_{n,\beta}^h \|^2
				+ \lambda \widehat{k}_n \| u^h_{n,\beta} \|_{L^4}^4
				\leq \frac{C}{\mu} \widehat{k}_n \| f_{n,\beta} \|^2.
			\end{split}
		\end{align}
        We employ the definition of $G$-norm in \eqref{G-norm} to \eqref{Stab-eq2} and 
        sum \eqref{Stab-eq2} over $n$ from $n=1$ to $m-1$ 
        \begin{align*}
        	&\frac{1+\theta}{4}\| u^h_{m} \|^2 +
        	\sum_{n=1}^{m-1} \Big[ \Big\| \sum_{\ell=0}^{2} a_{\ell}^{(n)} u^h_{n-1+ \ell} \Big\|^2 
        	 + \widehat{k}_n \Big( \frac{\mu}{2} \| \nabla u^h_{n,\beta}\|^2
        	+ \gamma \| w_{n,\beta}^h \|^2 + \lambda \| u^h_{n,\beta} \|_{L^4}^4 \Big) \Big]
        	\notag \\
        	&\leq C_{\beta}^{(m-1)} |\rho| \widehat{k}_{m-1} \big(  \| u^h_{m} \|^{2} + \| u^h_{m-1} \|^{2} + \| u^h_{m-2} \|^{2}  \big) + \sum_{n=1}^{m-2} \widehat{k}_n |\rho| \| u^h_{n,\beta} \|^2 
     		\\
        	&+ \frac{1}{4}(1+\theta) \| u^h_{1} \|^2
        	+ \frac{1}{4}(1+\theta) \| u^h_{0} \|^2 
        	+ \frac{C}{\mu} \sum_{n=0}^{m-1}  \widehat{k}_{n} \| f_{n,\beta} \|^2,  \notag 
        \end{align*}
        which implies \eqref{Stability of DLN inequality} by  discrete Gr\"onwall inequality \cite[p.369]{MR1043610}. 
    \end{proof}

    \section{Error analysis}
    \label{sec:sec4}
    Throughtout this section, we denote $u_{n}$, $w_{n}$, $p_{n}$ to be the true solutions of velocity $u$, auxiliary variable $w$, pressure $p$ at time $t_{n}$ respectively. 
    We decompose the error functions for velocity $e^u_n$, auxiliary variable $e^w_n$, and pressure $e^p_n$ at time $t_{n}$ as:
    \begin{align*}
		&e^u_n =u_n^h - u_n  = (u_n^h - \mathcal{S}_h u_n) - (u_n - \mathcal{S}_h u_n) :=\psi^u_n - \eta^u_n,
		\\
		&e^w_n =w_n^h - w_n = (w_n^h - \mathcal{S}_h w_n) - (w_n - \mathcal{S}_h w_n) := \psi^w_n - \eta^w_n ,
		\\
		&e^p_n =p_n^h - p_n=  (p_n^h - \mathcal{S}_h p_n) - (p_n - \mathcal{S}_h p_n) := \psi^p_n -\eta^p_n,
	\end{align*} 
	where  $\mathcal{S}_h$ represents the Stokes-type projection \eqref{Stokes-type-projection-defination}.
    We need consistency properties of the DLN method, as well as assumptions about the regularity of true solutions and the time step constraint.
    \begin{lemma}[consistency] \label{n beta int tn-1 tn+1}
		Suppose that $u(\cdot, t)$ is the mapping from $[0,T]$ to $H^{r}(D)$. Assuming that the mapping $u(\cdot, t)$ is third-order differentiable in time, then for any $\theta \in [0,1)$
		\begin{align*}
			&\big\| u_{n,\beta} - u(t_{n,\beta}) \big\|_{r}^2
			\leq
			C(\theta)(k_n+k_{n-1})^3 \int_{t_{n-1}}^{t_{n+1}}\|u_{tt}\|_{r}^2 \mathrm{d}t, \\
            &\Big\| \frac{u_{n,\alpha}}{\widehat{k}_n} - u_t(t_{n,\beta}) \Big\|_{r}^2
			\leq
			C(\theta) (k_n+k_{n-1})^3 \int_{t_{n-1}}^{t_{n+1}}\|u_{ttt}\|_{r}^2 \mathrm{d}t.
		\end{align*}
		For $\theta = 1$, the DLN method reduces to the midpoint rule. We have 
		\begin{align*}
			&\big\| u_{n,\beta} - u(t_{n,\beta}) \big\|_{r}^2
            = \Big\| \frac{u_{n+1} + u_{n}}{2} - u \Big(\frac{t_{n+1} + t_{n}}{2} \Big) \Big\|_{r}
			\leq
			C k_n^3 \int_{t_{n-1}}^{t_{n+1}}\|u_{tt}\|_{r}^2 \mathrm{d}t, \\
            &\Big\| \frac{u_{n,\alpha}}{\widehat{k}_n} - u_t(t_{n,\beta}) \Big\|_{r}^2
            = \Big\| \frac{u_{n+1} - u_{n}}{k_n} - u_{t} \Big(\frac{t_{n+1} + t_{n}}{2} \Big) \Big\|_{r}
			\leq C k_n^3 \int_{t_{n-1}}^{t_{n+1}}\|u_{ttt}\|_{r}^2 \mathrm{d}t.
		\end{align*}
	\end{lemma}
	\begin{proof}
		See \cite[Appendix B.]{CLPX2025_JSC} for the complete proof.
	\end{proof}

	\subsection{Error estimate for velocity in $L^{2}$-norm} 
	\ 
    \begin{assumption}\label{u_p_T_space}
		We assume that the exact solution $(u,w,p)$ in \eqref{ACEs} and the external body force $f$ satisfy the following regularity assumptions:
		\begin{align*}
			&u \in W^{3,2}\big([0,T]; [H^{r+2}(D)]^2 \big) \cap L^{\infty}\big([0,T]; [H^{r+2}(D)]^2 \big),
			\\
			&w \in W^{2,2}\big([0,T]; [H^{r+2}(D)]^2 \big) \cap L^{\infty}\big([0,T]; [H^{r+2}(D)]^2 \big),
			\\
			&p \in W^{2,2}\big([0,T]; H^{r+1}(D) \big) \cap L^\infty\big([0,T]; H^{r+1}(D)\big), \quad 
			f \in W^{2,2}\big([0,T]; [L^{2}(D)]^2 \big).
		\end{align*}
	\end{assumption}
    \begin{assumption}[Time step constraint] \label{time_condition_2}
        We assume that the weighted average time step $\widehat{k}_{n}$ satisfies: 
        for all $1 \leq n \leq M-1$ 
		\begin{align}
			\label{time_condition_eq2}
			C_{\beta}^{(n)} \widehat{k}_{n} \Big( |\rho| + \frac{ C^{\ast} \nu^4 }{\mu^{3}} \| \nabla u_{n,\beta} \|^{4} \Big)
			< \frac{1+\theta}{4},
		\end{align}
		where $C_{\beta}^{(n)}$ is defined in \eqref{time-cond-stab} and $C^{\ast} > 0$ denotes the constant $C$ in \eqref{b 0 1 1 1}.
	\end{assumption}

    \begin{theorem}\label{Intermediate conclusion}
		Suppose that Assumptions \ref{u_p_T_space} and \ref{time_condition_2} hold. For sufficiently small $k_n>0$, there exists a constant \( C \), independent of the mesh size \( h \) and the time step \( k_n \), such that the following error estimate holds:
		\begin{align} \label{proof of max u-u^h T-T^h error}
			\begin{split}
				&\max_{0 \leq n \leq M}  \| e_n^u \|^2
				+  \sum_{n=1}^{M-1} \widehat{k}_n \|\nabla e_{n,\beta}^u \|^2 
				+  \sum_{n=1}^{M-1} \widehat{k}_n \| e_{n,\beta}^{w} \|^2 
				\leq \mathcal{O}(h^{2r+2} + k_{\max}^4),
			\end{split}
		\end{align}
        where $\displaystyle k_{\max} = \max_{0 \leq n \leq M-1} \{ k_{n} \}$.
	\end{theorem}
    
    We first prove Lemma \ref{error estimate L2 lem} - \ref{truncation error estimate} and then finish the proof of 
    Theorem \ref{Intermediate conclusion}.
    \begin{lemma}\label{error estimate L2 lem}
		For $1 < m \leq M$, the following estimate holds:
		\begin{align} \label{error estimate L2 inequality}
			&\frac{1}{4}(1+\theta)\|\psi^u_m\|^2 + \frac{1}{4}(1-\theta)\|\psi^u_{m-1}\|^2 
			+ \sum_{n=1}^{m-1} \Big\| \sum_{\ell=0}^{2} a_{\ell}^{(n)} \psi^u_{n-1+\ell} \Big\|^2 \notag
			\\
			&+ \sum_{n=1}^{m-1}  \widehat{k}_n \Big\{
			\rho \|\psi^u_{n,\beta}\|^2
			+\mu \| \nabla \psi^u_{n,\beta}\|^2
			+\gamma \|\psi^w_{n,\beta}\|^2
			+ C \lambda  \|\psi^u_{n,\beta}\|^4_{L^4}
			\Big\} \notag
			\\
			\leq &
			\sum_{n=1}^{m-1} | (\eta^u_{n,\alpha} , \psi^u_{n,\beta})|
			+ \sum_{n=1}^{m-1} |\rho|\widehat{k}_n |(\eta^u_{n,\beta}, \psi^u_{n,\beta})|
			\\
			& + \sum_{n=1}^{m-1} \widehat{k}_n \nu |b(u_{n,\beta}, u_{n,\beta}, \psi^u_{n,\beta}) - b(u_{n,\beta}^h, u_{n,\beta}^h, \psi^u_{n,\beta})| 
			\notag
			\\
			& + \sum_{n=1}^{m-1} \lambda \widehat{k}_n |(|u_{n,\beta}|^2 u_{n,\beta} - |\mathcal{S}_h u_{n,\beta}|^2\mathcal{S}_h u_{n,\beta}, \psi^u_{n,\beta})|
			\notag
			\\
			& + \sum_{n=1}^{m-1}\widehat{k}_n |\tau(u,w,p,\psi^u_{n,\beta})|
			+ \frac{1}{4}(1+\theta)\|\psi^u_1\|^2 + \frac{1}{4}(1-\theta)\|\psi^u_{0}\|^2.
			\notag
		\end{align}
	\end{lemma}
    \begin{proof}
        The true solution pair $(u,w, p)$ in \eqref{variational-formula-1}-\eqref{variational-formula-3} satisfies the following weak formulation at time \(t_{n,\beta}\): for all $(v^h, \varphi^h, q^h) \in X_h \times  V_h \times Q_h$
	\begin{align} 
		\begin{split}\label{formulation 1}
			&\frac{1}{\widehat{k}_n}(u_{n,\alpha}, v^h) 
			+ \mu(\nabla u_{n,\beta}, \nabla v^h)
			+ \gamma (\nabla w_{n,\beta}, \nabla v^h)
			+ \nu b(u_{n,\beta}, u_{n,\beta}, v^h) 
			\\
			&~+\rho (u_{n,\beta}, v^h)
			+\lambda (|u_{n,\beta}|^2u_{n,\beta}, v^h)
			- (p_{n,\beta}, \nabla \cdot v^h) 
			=(f_{n,\beta},v^h)+  \tau(u,w,p, v^h),
		\end{split}	
		\\
		&(w_{n,\beta}, \varphi^h) - (\nabla u_{n,\beta},\nabla \varphi^h)
		=0, \label{formulation 2}
		\\
		&(\nabla \cdot u_{n,\beta} , q^h)=0,\label{formulation 3}
	\end{align}
	where $\tau(u,w,p, v^h)$ denotes the truncation error
	\begin{align} \label{truncation}
		\begin{split}
			\tau(u,w,p, v^h) = & \Big(\frac{1}{\widehat{k}_n} u_{n,\alpha} - u_t(t_{n,\beta}), v^h \Big) 
			+ \mu \big( \nabla (u_{n,\beta} -u(t_{n,\beta})), \nabla v^h \big) 
			\\
			&
            + \nu \big( b(u_{n,\beta}, u_{n,\beta}, v^h) - b(u(t_{n,\beta}), u(t_{n,\beta}), v^h) \big)
			\\
			&+ \rho \big( u_{n,\beta}-u(t_{n,\beta}), v^h \big)
			+ \gamma \big( \nabla (w_{n,\beta} - w(t_{n,\beta})), \nabla v^h \big)
			\\
			&+\lambda \big( |u_{n,\beta}|^2 u_{n,\beta}  - |u(t_{n,\beta})|^2 u(t_{n,\beta}), v^h \big),
			\\
			&- \big( p_{n,\beta} - p(t_{n,\beta}), \nabla \cdot v^h \big) 
            - \big(f_{n,\beta} - f(t_{n,\beta}),v^h \big).
		\end{split}
	\end{align}
	By subtracting \eqref{fully discrete formulations second 1} from \eqref{formulation 1} and 
    a suitable linear combination of \eqref{fully discrete formulations second 2}-\eqref{fully discrete formulations second 3} from \eqref{formulation 2}-\eqref{formulation 3},
	and further employing the incompressibility condition $\nabla \cdot u = 0$ along with the Stokes-type projection $\mathcal{S}_h$ defined in \eqref{Stokes-type-projection-defination}, 
	we have the corresponding error equations:
	\begin{align}
	\begin{split}\label{error equation 1}
	&\frac{1}{\widehat{k}_n}(\psi^u_{n,\alpha}, v^h) 
	+ \mu (\nabla \psi^u_{n,\beta}, \nabla v^h) 
	+\gamma (\nabla \psi^w_{n,\beta},\nabla v^h)
	+ \rho (\psi^u_{n,\beta},v^h) 
	- (\psi^p_{n,\beta}, \nabla \cdot v^h)
	\\
	&=  \frac{1}{\widehat{k}_n}(\eta^u_{n,\alpha}, v^h) 
	+\nu b(u_{n,\beta}, u_{n,\beta}, v^h)
	-\nu b(u_{n,\beta}^h, u_{n,\beta}^h, v^h)
	+ \rho (\eta^u_{n,\beta},v^h) 
	\\
	& \quad + \lambda (|u_{n,\beta}|^2 u_{n,\beta}  - |u_{n,\beta}^h|^2 u_{n,\beta}^h ,v^h)
	- \tau(u,w,p, v^h),
	\end{split}
	\\
	&(\psi^w_{n,\beta}, \varphi^h) - (\nabla \psi^u_{n,\beta}, \nabla \varphi^h)
	=0, \label{error equation 2}
	\\
	&(\nabla \cdot \psi^u_{n,\beta} , q^h)=0.
	\label{error equation 3}
	\end{align}
	We set \( v^h = \psi^u_{n,\beta} \) in \eqref{error equation 1},  \( \varphi^h = \psi^w_{n,\beta} \) in \eqref{error equation 2}, \( q^h = \psi^p_{n,\beta} \) in \eqref{error equation 3}, and multiply \eqref{error equation 1} by \( \widehat{k}_n \)
    \begin{align} \label{error equations n beta}
        \begin{split}
            &(\psi^u_{n,\alpha}, \psi^u_{n,\beta}) 
            + \mu \widehat{k}_n \| \nabla \psi^u_{n,\beta}\|^2
            +\gamma\widehat{k}_n(\nabla \psi^w_{n,\beta}, \nabla \psi^u_{n,\beta})
            +\rho\widehat{k}_n (\psi^u_{n,\beta}, \psi^u_{n,\beta})
            \\
            &
            + \lambda \widehat{k}_n (|u^h_{n,\beta}|^2 u^h_{n,\beta}  - |\mathcal{S}_h u_{n,\beta}|^2\mathcal{S}_h u_{n,\beta} ,\psi^u_{n,\beta})
            - \widehat{k}_n(\psi^p_{n,\beta}, \nabla \cdot \psi^u_{n,\beta})
            \\
            = &(\eta^u_{n,\alpha} , \psi^u_{n,\beta})
            + \rho\widehat{k}_n (\eta^u_{n,\beta}, \psi^u_{n,\beta})
            \\
            &+ \widehat{k}_n \nu b(u_{n,\beta}, u_{n,\beta}, \psi^u_{n,\beta})
            - \widehat{k}_n \nu b(u_{n,\beta}^h, u_{n,\beta}^h, \psi^u_{n,\beta})
            \\
            &+ \lambda \widehat{k}_n(|u_{n,\beta}|^2 u_{n,\beta} - |\mathcal{S}_h u_{n,\beta}|^2\mathcal{S}_h u_{n,\beta}, \psi^u_{n,\beta})
            - \widehat{k}_n \tau(u,w,p, \psi^u_{n,\beta}),
        \end{split}
        \\
        \begin{split}
            (\psi^w_{n,\beta},\psi^w_{n,\beta}) 
            = (\nabla \psi^u_{n,\beta}, \nabla \psi^w_{n,\beta}),
        \end{split}
        \\
        \begin{split}
            (\nabla \cdot \psi^u_{n,\beta}, \psi^p_{n,\beta}) = 0.
        \end{split}
    \end{align}
    We combine \eqref{error equations n beta} and utilize \eqref{vector_norm_inequality_1}, 
    \eqref{G-stable-lemma} to derive the following estimate:
		\begin{align*} 
			&\Big\|
			\begin{array}{cc}
				\psi^u_{n+1} \\
				\psi^u_{n}\\
			\end{array} \Big\|^2_{G(\theta )}
			-
			\Big\|
			\begin{array}{cc}
				\psi^u_{n} \\
				\psi^u_{n-1}\\
			\end{array} \Big\|^2_{G(\theta )}
			+ \Big\| \sum_{\ell=0}^{2} a_{\ell}^{(n)} \psi^u_{n-1+\ell} \Big\|^2
			\\
			&+ \rho \widehat{k}_n \|\psi^u_{n,\beta}\|^2
			+ \mu \widehat{k}_n \| \nabla \psi^u_{n,\beta}\|^2
			+ \gamma \widehat{k}_n \|\psi^w_{n,\beta}\|^2
			+ C \lambda \widehat{k}_n  \|\psi^u_{n,\beta}\|^4_{L^4}
			\\
			\leq &
			|(\eta^u_{n,\alpha} , \psi^u_{n,\beta})|
			+ |\rho|\widehat{k}_n |(\eta^u_{n,\beta}, \psi^u_{n,\beta})|
			+ \lambda \widehat{k}_n|(|u_{n,\beta}|^2 u_{n,\beta} - |\mathcal{S}_h u_{n,\beta}|^2\mathcal{S}_h u_{n,\beta}, \psi_{n,\beta}^u )|
			\\
			&+ \widehat{k}_n \nu (b(u_{n,\beta}, u_{n,\beta}, \psi^u_{n,\beta}) - b(u_{n,\beta}^h, u_{n,\beta}^h, \psi^u_{n,\beta}))
			+\widehat{k}_n |\tau(u,w,p,\psi^u_{n,\beta})|.
		\end{align*}
        We sum the above inequality from $n=1$ to $m-1$ and derive \eqref{error estimate L2 inequality} by the definition of the $G$-norm in \eqref{G-norm}.
    \end{proof}

    We address the terms on the right side on inequality \eqref{error estimate L2 inequality} by Lemma \ref{derivation term} - \ref{truncation error estimate}  
    \begin{lemma}\label{derivation term}
		Suppose Assumption \ref{u_p_T_space} holds. For $ 1 \leq n \leq M-1 $, there exists a constant $ C \geq 0 $ independent of mesh size $h$ and the time step $k_n$, such that 
		\begin{align} \label{derivation-term-conclusion}
            \begin{split}
			& |(\eta^u_{n,\alpha}, \psi^u_{n,\beta})| + |\rho| \widehat{k}_n |(\eta^u_{n,\beta}, \psi^u_{n,\beta})| \\
			\leq &  C(\theta) \frac{h^{2r+4}}{\mu} \int_{t_{n-1}}^{t_{n+1}} \big( \| u_t \|_{r+2}^2 + \| w_t \|_{r+2}^2 + \| p_t \|_{r+1}^2\big) \, \mathrm{d}t 
            +  \frac{\mu}{16}\widehat{k}_n \| \psi^u_{n,\beta} \|^2 \\
            &+ \frac{ C(\theta) \widehat{k}_n \rho^2 h^{2r+4}}{\mu} \big( \| u \|_{\infty,r+2} +  \| w \|_{\infty,r+2} + \| p \|_{\infty,r+1} \big)^{2}.
            \end{split}
		\end{align}
	\end{lemma}
    \begin{proof}
        We first prove the case $\theta \in [0,1)$. By Cauchy-Schwarz inequality, approximation of Stokes-type projection $\mathcal{S}_h$ in \eqref{Stokes-type projection}
        \begin{align} \label{derivation-term-eq1}
            \begin{split}
			|(\eta^u_{n,\alpha}, \psi^u_{n,\beta})| 
			&\leq C h^{r+2} \big( \| u_{n,\alpha} \|_{r+2} + \| w_{n,\alpha} \|_{r+2} 
            + \| p_{n,\alpha} \|_{r+1} \big) \| \nabla \psi^u_{n,\beta} \|.
            \end{split}
		\end{align}
        By fundamental theorem of Calculus, triangle inequality and H\"older's inequality
        \begin{align} \label{derivation-term-eq2}
            \begin{split}
            \| u_{n,\alpha} \|_{r+2} =&
            \Big\| \frac{1}{2}(\theta - 1) u_{n-1} - \theta u_{n} + \frac{1}{2}(\theta + 1) u_{n+1} \Big\|_{r+2}    \\
            \leq & \frac{1}{2}(\theta + 1) \Big\| \int_{t_n}^{t_{n+1}} u_t \, \mathrm{d}t \Big\|_{r+2} + \frac{1}{2}(1 - \theta) \Big\| \int_{t_{n-1}}^{t_n} u_t \, \mathrm{d}t \Big\|_{r+2}  \\
            \leq & C(\theta) k_{n}^{\frac{1}{2}} \Big( \int_{t_n}^{t_{n+1}} \| u_t \|_{r+2}^2 \, \mathrm{d}t \Big)^{\frac{1}{2}} + C(\theta) k_{n-1}^{\frac{1}{2}} \Big( \int_{t_{n-1}}^{t_n} \| u_t \|_{r+2}^2 \, \mathrm{d}t \Big)^{\frac{1}{2}},  \\
            \| w_{n,\alpha} \|_{r+2} 
            \leq & C(\theta) k_{n}^{\frac{1}{2}} \Big( \int_{t_n}^{t_{n+1}} \| w_t \|_{r+2}^2 \, \mathrm{d}t \Big)^{\frac{1}{2}} + C(\theta) k_{n-1}^{\frac{1}{2}} \Big( \int_{t_{n-1}}^{t_n} \| w_t \|_{r+2}^2 \, \mathrm{d}t \Big)^{\frac{1}{2}},  \\
            \| p_{n,\alpha} \|_{r+1} 
            \leq & C(\theta) k_{n}^{\frac{1}{2}} \Big( \int_{t_n}^{t_{n+1}} \| p_t \|_{r+1}^2 \, \mathrm{d}t \Big)^{\frac{1}{2}} + C(\theta) k_{n-1}^{\frac{1}{2}} \Big( \int_{t_{n-1}}^{t_n} \| p_t \|_{r+1}^2 \, \mathrm{d}t \Big)^{\frac{1}{2}}. 
            \end{split}
		\end{align}
        We combine \eqref{derivation-term-eq1}, \eqref{derivation-term-eq2} and use Young's inequality to have  
        \begin{equation} \label{derivation-term-eq3}
			|(\eta^u_{n,\alpha}, \psi^u_{n,\beta})| 
            \leq C(\theta) \frac{h^{2r+4}}{\mu} \int_{t_{n-1}}^{t_{n+1}} \big( \| u_t \|_{r+2}^2 + \| w_t \|_{r+2}^2 + \| p_t \|_{r+1}^2\big) \, \mathrm{d}t 
            +  \frac{\mu}{32}\widehat{k}_n \| \psi^u_{n,\beta} \|^2.
		\end{equation}
        By Cauchy-Schwarz inequality and Young's inequality
		\begin{align} 
				|\rho| \widehat{k}_n |(\eta^u_{n,\beta}, \psi^u_{n,\beta})| 
                \leq& C \rho^2 \mu^{-1} \widehat{k}_n \|\eta^u_{n,\beta}\|^2 
				+ \frac{\mu}{32}\widehat{k}_n \|\psi^u_{n,\beta}\|^2 
                \label{derivation-term-eq4} \\
                \leq& \frac{C(\theta) \widehat{k}_n \rho^2 h^{2r+4}}{\mu} \big( \| u \|_{\infty,r+2} +  \| w \|_{\infty,r+2} + \| p \|_{\infty,r+1} \big)^{2} 
                + \frac{\mu}{32}\widehat{k}_n \|\psi^u_{n,\beta}\|^2. \notag 
		\end{align}
        We combine \eqref{derivation-term-eq3} and \eqref{derivation-term-eq4} to 
        obtain \eqref{derivation-term-conclusion}.
        If $\theta = 1$, $\widehat{k}_n = k_{n}$ and \eqref{derivation-term-eq2} becomes  
        \begin{align*}
            \| u_{n,\alpha} \|_{r+2} =
            \big\| u_{n+1} - u_{n} \big\|_{r+2} \leq k_{n}^{\frac{1}{2}} \Big( \int_{t_n}^{t_{n+1}} \| u_t \|_{r+2}^2 \, \mathrm{d}t \Big)^{\frac{1}{2}}.
        \end{align*}
        Then the proof is very similar to the case of $\theta \in [0,1)$.
    \end{proof}

    \begin{lemma}\label{analytical b - numerical b}
		Under Assumption \ref{u_p_T_space}, the following bound holds:
		\begin{align} 
			&\nu \big( b(u_{n,\beta}, u_{n,\beta}, \psi^u_{n,\beta}) - b(u_{n,\beta}^h, u_{n,\beta}^h, \psi^u_{n,\beta}) \big) 
            \label{analytical b - numerical b-conclusion} \\
			\leq &\frac{ C^{\ast} \nu^{4}}{\mu^{3}}  \|\nabla u_{n,\beta}\|^4 
             \|\psi^u_{n,\beta}\|^2 
             + \frac{C(\theta) \nu^{2} h^{2r+4}}{\mu} \| u \|_{\infty,2}^{2} 
             \big( \| u \|_{\infty,r+2} + \| w \|_{\infty,r+2} 
             + \| p \|_{\infty,r+1} \big)^{2} \notag \\
            &+ \frac{C(\theta) \nu^{2} h^{2r+2}}{\mu} \| \nabla u_{n,\beta}^h \|^{2}
            \big( \| u \|_{\infty,r+2} + \| w \|_{\infty,r+2} 
            + \| p \|_{\infty,r+1} \big)^{2} + \frac{\mu}{16} \|\nabla \psi^u_{n,\beta}\|^2. \notag 
		\end{align}
        for $ 1 \leq n \leq M-1 $. Here $C^{\ast} > 0$ denotes the constant $C$ in \eqref{b 0 1 1 1}.
	\end{lemma}
    \begin{proof}
        By the skew-symmetric property of $b(\cdot, \cdot, \cdot)$
		\begin{align}
			\begin{split}\label{b-b inequality}
				& \nu b(u_{n,\beta}, u_{n,\beta}, \psi^u_{n,\beta}) - \nu b(u_{n,\beta}^h, u_{n,\beta}^h, \psi^u_{n,\beta}) \\
				=& \nu b(u_{n,\beta} - u_{n,\beta}^h, u_{n,\beta}, \psi^u_{n,\beta}) 
                + \nu b(u_{n,\beta}^h, u_{n,\beta} - u_{n,\beta}^h, \psi^u_{n,\beta}) \\
				=& \nu b(\eta^u_{n,\beta}, u_{n,\beta}, \psi^u_{n,\beta}) - \nu b(\psi^u_{n,\beta}, u_{n,\beta}, \psi^u_{n,\beta}) + \nu b(u_{n,\beta}^h, \eta^u_{n,\beta}, \psi^u_{n,\beta}) 
			\end{split}
		\end{align}
        We utilize the \eqref{b 0 2 1}, Poincar\'e inequality and Young's inequality to obtain:
		\begin{align} 
			\nu b(\eta^u_{n,\beta}, u_{n,\beta}, \psi^u_{n,\beta})
			&\leq C \nu \|\eta^u_{n,\beta}\| \|u_{n,\beta}\|_2 \|\psi^u_{n,\beta}\|_1 
            \label{b-term1} \\
			&\leq C \nu^{2} \mu^{-1} \|\eta^u_{n,\beta}\|^2 \|u_{n,\beta}\|_2^2 + \frac{1}{48}\mu \|\nabla \psi^u_{n,\beta}\|^2. \notag 
		\end{align}
        We make use of \eqref{b 0 1 1 1} and Young's inequality
		\begin{align} 
			- \nu b(\psi^u_{n,\beta}, u_{n,\beta}, \psi^u_{n,\beta}) 
			&\leq C \nu \|\psi^u_{n,\beta}\|^{\frac{1}{2}}\|\nabla \psi^u_{n,\beta}\|^{\frac{1}{2}} \|\nabla u_{n,\beta}\| \|\nabla \psi^u_{n,\beta}\| 
            \label{b-term2} \\
			&\leq C \mu^{-3} \nu^{4} \|\psi^u_{n,\beta}\|^2 \|\nabla u_{n,\beta}\|^4 + \frac{1}{48}\mu \|\nabla \psi^u_{n,\beta}\|^2. \notag 
		\end{align}
		By \eqref{b 1 1 1} and Young's inequality
		\begin{align}
			\nu b(u_{n,\beta}^h, \eta^u_{n,\beta}, \psi^u_{n,\beta}) 
			\leq& C \nu \| \nabla u_{n,\beta}^h \| \| \nabla \eta^u_{n,\beta} \| \| \nabla \psi^u_{n,\beta} \| 
			\label{b-term3} \\
			\leq& C \nu^2 \mu^{-1}\| \nabla u_{n,\beta}^h \|^{2}  \| \nabla \eta^u_{n,\beta} \|^{2} + \frac{1}{48}\mu \|\nabla \psi^u_{n,\beta}\|^2. \notag 
		\end{align}
        We combine \eqref{b-b inequality} - \eqref{b-term3} and use approximation of Stokes-type projection in \eqref{Stokes-type projection} again to 
        to derive \eqref{analytical b - numerical b-conclusion}.
    \end{proof}

    \begin{lemma}\label{cubic term}
		Under Assumption \ref{u_p_T_space}, the following bound holds 
		\begin{align}
				&\lambda |(|u_{n,\beta}|^2 u_{n,\beta} - |\mathcal{S}_h u_{n,\beta}|^2\mathcal{S}_h u_{n,\beta},\psi_{n,\beta}^u )|
				\label{cubic term inequality} \\
				\leq & \frac{C(\theta) \lambda^2 h^{2r+2}}{\mu} \Big[ \| u \|_{L^{\infty}(L^{4})}^{4} 
                + h^{4r+4} \big( \| u \|_{\infty,r+2} 
                + \| w \|_{\infty,r+2} + \| p \|_{\infty,r+1} \big)^{4} \Big] \times \notag \\
				&\times \big( \| u \|_{\infty,r+2} 
                + \| w \|_{\infty,r+2} + \| p \|_{\infty,r+1} \big)^{2}
				+ \frac{1}{16} \mu \|\nabla \psi^u_{n,\beta}\|^2. \notag 
		\end{align}
        for $1 \leq n \leq M-1$.
    \end{lemma}
    \begin{proof}
        By \eqref{vector_norm_inequality_2}, H\"older's inequality, Sobolev imbedding inequality, Poincar\'e inequality and triangle inequality 
        \begin{align}\label{cubic term inequality 1}
            &|(|u_{n,\beta}|^2 u_{n,\beta} - |\mathcal{S}_h u_{n,\beta}|^2\mathcal{S}_h u_{n,\beta}, \psi_{n,\beta}^u)| \notag
            \\
            \leq
            & C \Big(
            \int_{D} (|u_{n,\beta}|^2 u_{n,\beta} -|\mathcal{S}_h u_{n,\beta}|^2 \mathcal{S}_h u_{n,\beta})^{\frac{4}{3}}\mathrm{d}x
            \Big)^\frac{3}{4}
            \Big(
            \int_{D}| \psi_{n,\beta}^u|^4 \mathrm{d}x
            \Big)^\frac{1}{4} \notag
            \\
            \leq
            &
            C \bigg[
            \Big(
            \int_{D} (|u_{n,\beta}| + |\mathcal{S}_h u_{n,\beta}| )^4 \mathrm{d}x
            \Big)^\frac{2}{3}
            \Big(
            \int_{D} |u_{n,\beta} - \mathcal{S}_h u_{n,\beta}|^4 \mathrm{d}x
            \Big)^\frac{1}{3}
            \bigg]^\frac{3}{4}
            \|\psi_{n,\beta}^u\|_{L^4}
            \\
            \leq
            & C 
            \Big(
            \int_{D} (|u_{n,\beta}| + |\mathcal{S}_h u_{n,\beta}| )^4 \mathrm{d}x
            \Big)^\frac{1}{2} \|u_{n,\beta}-\mathcal{S}_h u_{n,\beta}\|_{L^4}
            \|\psi_{n,\beta}^u\|_{L^4} \notag
            \\ 
            \leq
            & C (\|u_{n,\beta}\|_{L^4(D)}^2 +\|\mathcal{S}_h u_{n,\beta}\|_{L^{4}(D)}^2 )
            \|\nabla \eta^u_{n,\beta}\| \|\nabla \psi^u_{n,\beta} \| \notag \\
            \leq& C (\|u_{n,\beta}\|_{L^4(D)}^2 + \|\eta^{u}_{n,\beta}\|_{1}^2 ) \|\nabla \eta^u_{n,\beta}\| \|\nabla \psi^u_{n,\beta}\|. \notag 
        \end{align}
        We apply Young's inequality and approximation in \eqref{Stokes-type projection} to \eqref{cubic term inequality 1} and  
        obtain \eqref{cubic term inequality}.
    \end{proof} 

    \begin{lemma}\label{analytical b - exact b with t n beta}
		Under Assumption \ref{u_p_T_space}, the following bound holds 
		\begin{align*}
			&\nu \big(
			b(u_{n,\beta}, u_{n,\beta}, \psi^u_{n,\beta}) - b(u(t_{n,\beta}), u(t_{n,\beta}), \psi^u_{n,\beta})
			\big)
			\\
			\leq&
			\frac{\mu}{32}  \| \nabla \psi^u_{n,\beta} \|^2
            + \frac{C(\theta) \nu^2}{\mu} \|u\|_{\infty,2}^2 k_{\rm{max}}^3 
            \int_{t_{n-1}}^{t_{n+1}} \| \nabla u_{tt}\|^2 \mathrm{d}t
		\end{align*}
        for $1 \leq n \leq M-1$.
	\end{lemma}
    \begin{proof}
		%
		We reformulate the difference between the two nonlinear terms as follows:
		\begin{align}
			\begin{split}\label{b-b n beta inequality}
				&\nu \big(
				b(u_{n,\beta}, u_{n,\beta}, \psi^u_{n,\beta}) - b(u(t_{n,\beta}), u(t_{n,\beta}), \psi^u_{n,\beta})
				\big)
				\\
				=&\nu b \big(u_{n,\beta} - u(t_{n,\beta}), u_{n,\beta}, \psi^u_{n,\beta} \big) 
				+ \nu  b \big(u(t_{n,\beta}), u_{n,\beta} - u(t_{n,\beta}), \psi^u_{n,\beta} \big).
			\end{split}
		\end{align}
		We apply \eqref{b 1 1 1}, Lemma \ref{n beta int tn-1 tn+1} and Young's inequality to \eqref{b-b n beta inequality}
		\begin{align} \label{b-b n beta inequality second}
			\begin{split}
				&\nu \big( b(u_{n,\beta} - u(t_{n,\beta}), u_{n,\beta}, \psi^u_{n,\beta}) \big)
				+ \nu \big( b(u(t_{n,\beta}), u_{n,\beta} - u(t_{n,\beta}), \psi^u_{n,\beta}) \big)
				\\
				\leq&
				C\nu (\| \nabla u_{n,\beta}\| + \| \nabla u(t_{n,\beta})\| )  \| \nabla (u_{n,\beta} - u(t_{n,\beta}))\| \| \nabla \psi^u_{n,\beta}\| 		
				\\ 
				\leq&
				\frac{\mu}{32}  \| \nabla \psi^u_{n,\beta} \|^2
				+ C \nu^2 \mu^{-1}
				\|u\|_{\infty,2}^2
				\|\nabla (u_{n,\beta} - u(t_{n,\beta}))\|^2
				\\
				\leq&
			    \frac{\mu}{32}  \| \nabla \psi^u_{n,\beta} \|^2
				+ C(\theta) \nu^2 \mu^{-1}
				\|u\|_{\infty,2}^2 (k_{n-1}+k_n)^3 
				\int_{t_{n-1}}^{t_{n+1}} \| \nabla u_{tt}\|^2 \mathrm{d}t,
			\end{split}
		\end{align}
		which completes the proof.
	\end{proof}

    \begin{lemma}\label{truncation error estimate}
        Under Assumption \ref{u_p_T_space},
		the truncation term $|\tau(u,w,p,\psi^u_{n,\beta})|$ in \eqref{truncation} has following error estimate
		\begin{align} 
				&|\tau(u,w,p,\psi^u_{n,\beta})| 
				\label{truncation error estimate equation}   \\
				&\leq \frac{\mu}{4} \|\nabla \psi^u_{n,\beta}\|^2
                + \frac{C(\theta)}{\mu} \big( \nu^2 \| u \|_{\infty,2}^2 + \lambda^2 \| u \|_{L^{\infty}(L^{4})}^{4} \big)
                k_{\rm{max}}^3 \int_{t_{n-1}}^{t_{n+1}}\|\nabla u_{tt}\|^2  \mathrm{d}t \notag \\
				&+ \frac{C(\theta) k_{\rm{max}}^3}{\mu} \int_{t_{n-1}}^{t_{n+1}} \!\big(
				\|u_{ttt}\|^2 \!+\! \mu^2 \|\nabla u_{tt}\|^2 \!+\! \rho^2 \|u_{tt}\|^2 \!+\! \gamma^{2} \|\nabla w_{tt}\|^2 \!+\! \| p_{tt}\|^2 \!+\! \|f_{tt}\|^2
				\big)	 \mathrm{d}t.   \notag 
		\end{align}
        for $1 \leq n \leq m-1$,
	\end{lemma}
    \begin{proof}
        By similar argument to \eqref{cubic term inequality 1}, 
        \begin{align}\label{cubic term truncation inequality 1}
            \begin{split}
                &\lambda|(|u_{n,\beta}|^2 u_{n,\beta} - |u(t_{n,\beta})|^2u(t_{n,\beta}), \psi^u_{n,\beta})|
                \\
                \leq
                & C \lambda \big( \| u_{n,\beta} \|_{L^{4}}^2 + \| u(t_{n,\beta}) \|_{L^{4}}^2  \big) \big\| \nabla \big( u_{n,\beta} - u(t_{n,\beta}) \big) \big\| \| \nabla \psi^u_{n,\beta} \|
            \end{split}
        \end{align}
        We apply Young's inequality to \eqref{cubic term truncation inequality 1}
        \begin{align} 
            &\lambda|(|u_{n,\beta}|^2 u_{n,\beta} - |u(t_{n,\beta})|^2u(t_{n,\beta}), \psi^u_{n,\beta})|
            \label{cubic term truncation inequality 2} \\
            \leq
            &C(\theta) \frac{\lambda^2}{\mu} \big( \| u_{n,\beta} \|_{L^{4}}^4 + \| u(t_{n,\beta}) \|_{L^{4}}^4  \big) 
            (k_{n-1}+k_{n})^3 \int_{t_{n-1}}^{t_{n+1}}\|\nabla u_{tt}\|^2  \mathrm{d}t
            + \frac{\mu}{32} \| \nabla \psi^u_{n,\beta}\|^2. \notag 
        \end{align}
        For the remaining truncation terms in \eqref{truncation}, 
        we utilize Cauchy-Schwarz, Poincar\'e inequality, Young's inequalities, along with Lemma \ref{n beta int tn-1 tn+1}
        \begin{align}
			\Big( \frac{1}{\widehat{k}_n} u_{n,\alpha} - u_t(t_{n,\beta}), \psi_{n,\beta}^u \Big) 
			\leq& C \mu^{-1} \Big\| \frac{1}{\widehat{k}_n} u_{n,\alpha} - u_t(t_{n,\beta}) \Big\|^2 + \frac{\mu}{32} \| \nabla \psi_{n,\beta}^u\|^2
			\label{cubic term truncation inequality 3} \\
			\leq&
			\frac{C(\theta)}{\mu} (k_{n-1}+k_{n})^3 \int_{t_{n-1}}^{t_{n+1}}\|u_{ttt}\|^2  \mathrm{d}t
			+ \frac{\mu}{32} \| \nabla \psi_{n,\beta}^u\|^2,
			\notag \\
			\mu (\nabla (u_{n,\beta} -u(t_{n,\beta})), \nabla \psi_{n,\beta}^u) 
			\leq& 
			C(\theta) \mu (k_{n-1}+k_{n})^3 \int_{t_{n-1}}^{t_{n+1}}\|\nabla u_{tt}\|^2  \mathrm{d}t
			+\frac{\mu}{32}\| \nabla \psi_{n,\beta}^u\|^2,
			\notag \\
			\rho (u_{n,\beta}-u(t_{n,\beta}), \psi_{n,\beta}^u)
			\leq &
			C(\theta) \frac{\rho^2}{\mu} (k_{n-1}+k_{n})^3 \int_{t_{n-1}}^{t_{n+1}}\|u_{tt}\|^2  \mathrm{d}t
			+ \frac{\mu}{32} \| \nabla \psi_{n,\beta}^u\|^2,
			\notag \\
			\gamma(\nabla (w_{n,\beta} - w(t_{n,\beta})), \nabla\psi_{n,\beta}^u)
			\leq & 
			C(\theta) \frac{\gamma^2}{\mu} (k_{n-1}+k_{n})^3 \int_{t_{n-1}}^{t_{n+1}}\|\nabla w_{tt}\|^2  \mathrm{d}t
			+\frac{\mu}{32}  \| \nabla \psi_{n,\beta}^u\|^2,
			\notag \\
			(p_{n,\beta} - p(t_{n,\beta}), \nabla \cdot \psi_{n,\beta}^u)
			\leq & 
			\frac{C(\theta)}{\mu} (k_{n-1}+k_{n})^3 \int_{t_{n-1}}^{t_{n+1}}\| p_{tt}\|^2  \mathrm{d}t
			+\frac{\mu}{32} \| \nabla \psi_{n,\beta}^u\|^2,
			\notag \\
			(f_{n,\beta} - f(t_{n,\beta}),\psi_{n,\beta}^u)
			\leq &  
			\frac{C(\theta)}{\mu} (k_{n-1}+k_{n})^3 \int_{t_{n-1}}^{t_{n+1}}\|f_{tt}\|^2  \mathrm{d}t
			+\frac{\mu}{32} \|\nabla \psi_{n,\beta}^u\|^2. \notag 
		\end{align}
        We combine \eqref{cubic term truncation inequality 2}, \eqref{cubic term truncation inequality 3} and Lemma \ref{analytical b - exact b with t n beta} to have \eqref{truncation error estimate equation}.
    \end{proof}

    Now we prove Theorem \ref{Intermediate conclusion}. 
    \begin{proof}
    We combine Lemma \ref{derivation term} - \ref{truncation error estimate} and use the approximation of Stokes-type projection 
    in \eqref{Stokes-type projection}
    \begin{align*}
			&\frac{1}{4}(1+\theta)\|\psi^u_m\|^2 
			+\frac{ \mu}{2} \sum_{n=1}^{m-1}\widehat{k}_n \| \nabla \psi^u_{n,\beta}\|^2
			+\gamma \sum_{n=1}^{m-1}\widehat{k}_n\| \psi^w_{n,\beta}\|^2 
			\\
			&\leq \! C_{\beta}^{(m-1)} \widehat{k}_{m-1}   \Big( |\rho| +
            \frac{ C^{\ast} \nu^{4}}{\mu^{3}}  \|\nabla u_{m-1,\beta}\|^4 \Big) 
            \big( \|\psi^u_{m}\|^2 + \|\psi^u_{m-1}\|^2 + \|\psi^u_{m-2}\|^2 \big)  
            \\
            &+ \sum_{n=1}^{m-2} \Big( |\rho| \!+\!
            \frac{ C^{\ast} \nu^{4}}{\mu^{3}}  \|\nabla u_{n,\beta}\|^4 \Big)  \widehat{k}_n \|\psi^u_{n,\beta}\|^2 
            \!+\! \frac{C(\theta) h^{2r+4}}{\mu} \big( \| u_t \|_{2,r+2}^2 \!+\! \| w_t \|_{2,r+2}^2 
            \!+\! \| p_t \|_{2,r+1}^2 \big) \\
            &+ \frac{ C(\theta) T h^{2r+4}}{\mu} \big( \rho^2 + \nu^{2} \| u \|_{\infty,2}^{2}  \big) \big( \| u \|_{\infty,r+2} +  \| w \|_{\infty,r+2} + \| p \|_{\infty,r+1} \big)^{2}
            \\
            &+ \frac{C(\theta) \nu^{2} h^{2r+2}}{\mu^{2}}
            \big( \| u \|_{\infty,r+2} + \| w \|_{\infty,r+2} 
            + \| p \|_{\infty,r+1} \big)^{2} \Big( \mu \sum_{n=1}^{m-1} \widehat{k}_{n} \| \nabla u_{n,\beta}^{h} \|^{2} \Big)\\
            &\frac{C(\theta) T \lambda^2 h^{2r+2}}{\mu} \Big[ \| u \|_{L^{\infty}(L^{4})}^{4} 
            + h^{4r+4} \big( \| u \|_{\infty,r+2} 
            + \| w \|_{\infty,r+2} + \| p \|_{\infty,r+1} \big)^{4} \Big] \times \notag \\
            &\times \!\big( \| u \|_{\infty,r+2} 
            \!+\! \| w \|_{\infty,r+2} + \| p \|_{\infty,r+1} \big)^{2} 
            \!+\! \frac{C(\theta)}{\mu} \big( \nu^2 \|u\|_{\infty,2}^2 \!+\! \lambda^2 \| u \|_{L^{\infty}(L^{4})}^{4} \big) k_{\rm{max}}^4 
            \| \nabla u_{tt} \|_{2,0}^{2}
            \\
            & + \frac{C(\theta) k_{\rm{max}}^4}{\mu} \big(\|u_{ttt}\|_{2,0}^2 
            + \mu^2 \|\nabla u_{tt}\|_{2,0}^2 + \rho^2 \|u_{tt}\|_{2,0}^2 
            + \gamma^{2} \|\nabla w_{tt}\|_{2,0}^2 + \| p_{tt}\|_{2,0}^2 
            + \|f_{tt}\|_{2,0}^2 \big)
            \\
			&+ \frac{1}{4}(1+\theta)\|\psi^u_1\|^2 + \frac{1}{4}(1-\theta)\|\psi^u_{0}\|^2, 
		\end{align*}
		where $C_{\beta}^{(n)}>0$ is the same as that in \eqref{time-cond-stab} and $C^{\ast}>0$ is the constant $C$ in \eqref{b 0 1 1 1}.
        We assume that $\|\psi^u_1\|^2$ and $\|\psi^u_0\|^2$ are of order $h^{2r+2}$ and apply \eqref{Stability of DLN inequality},  discrete Gr\"onwall inequality \cite[p.369]{MR1043610} along with time step restriction \eqref{time_condition_eq2} in 
        Assumption \ref{time_condition_2} to the above inequality 
        \begin{align} \label{error estimate psi gamma inequality}
			\|\psi^u_m\|^2 + \sum_{n=1}^{m-1} \widehat{k}_n \|\nabla \psi^u_{n,\beta}\|^2
			+ \sum_{n=1}^{m-1} \widehat{k}_n \| \psi^w_{n,\beta}\|^2 \leq \mathcal{O}(h^{2r+2} + k_{\max}^4).
		\end{align}
        By triangle inequality, approximation of $ \mathcal{S}_h u_n$ 
        in \eqref{Stokes-type projection} and \eqref{error estimate psi gamma inequality}
		\begin{align}\label{bound on the velocity error}
			\begin{split}
				&\max_{0 \leq n \leq M}  \| e_n^u \|^2  
				\leq \max_{0 \leq n \leq M} \|\eta^u_n\|^2 + \max_{0 \leq n \leq M} \|\psi^u_n\|^2 \leq \mathcal{O}(h^{2r+2}+ k_{\max}^4), 
			\end{split}
		\end{align}
		Similarly, we apply triangle inequality, Lemma \ref{n beta int tn-1 tn+1} and approximation of $ \mathcal{S}_h u_n$ 
        in \eqref{Stokes-type projection} and \eqref{error estimate psi gamma inequality}
		\begin{align}\label{nabla u-u t n beta inequality}
			\begin{split}
				&\sum_{n=1}^{M-1} \widehat{k}_n \|\nabla e_{n,\beta}^u \|^2 
                + \sum_{n=1}^{M-1} \widehat{k}_n \| e_{n,\beta}^{w} \|^2 \\
				\leq& \sum_{n=1}^{M-1} \widehat{k}_n \big( \|\nabla \psi^u_{n,\beta}\|^2 
                + \|\nabla \eta^u_{n,\beta}\|^2 \big)
                + \sum_{n=1}^{M-1} \widehat{k}_n \big( \|\nabla \psi^w_{n,\beta}\|^2 
                + \|\nabla \eta^w_{n,\beta}\|^2 \big) \\
				\leq& \mathcal{O}(h^{2r+2} + k_{\max}^4).
			\end{split}
		\end{align}
        We combine \eqref{bound on the velocity error} and \eqref{nabla u-u t n beta inequality} to achieve \eqref{proof of max u-u^h T-T^h error}.
	\end{proof}

	\subsection{Error estimate for velocity in $H^1$-norm} \ 
    To carry out the error estimate for velocity $u$ in $H^1$-norm, we need two more restrictions about uniform mesh diameter and time step size
    \begin{align}
        \label{time-diameter-condition}
        k_{\max}^{3} \leq h, \quad h^{2r+2} \leq k_{\min} = \min_{0 \leq n \leq M-1} \{ k_{n} \},
        \quad \frac{k_{\max}}{k_{\min}} \leq C
    \end{align}
	for some $C >0$. 
	\begin{theorem}\label{Intermediate conclusion 2}
		Suppose the assumption of \ref{u_p_T_space}, \ref{time_condition_2} and 
        conditions \eqref{time-diameter-condition} hold. 
        If the time step size and uniform mesh diameter are sufficiently small such that for all $1 \leq n \leq M-1$
        \begin{align}
            \label{time-condition-H1}
            C(\theta) \Big[ \frac{\nu^2 }{h} \| \nabla e_{n,\beta}^u \|^{2} 
            + \lambda^2 \big( \|u_{n,\beta}\|_{1}^4 + \| \nabla e_{n,\beta} \|_{1}^4 \big) \Big]
            \widehat{k}_{n} \leq \frac{\mu(1+\theta)}{4}, 
        \end{align}
        for some $C(\theta)>0$, then Scheme \ref{fully discrete formulations scheme second} has the following error estimates 
		\begin{equation} \label{proof of max nabla u-u^h T-T^h error}
		\max_{0 \leq n \leq M}  \|e_{n}^{u} \|_1^{2}
		+ \max_{0 \leq n \leq M}  \|e_{n}^{w} \|^{2}
        + \sum_{n=1}^{M-1} \widehat{k}_n \|\widehat{k}_n^{-1} e^u_{n,\alpha}\|^2
		\leq
		\mathcal{O}(h^{2r+2}  + k_{\max}^4).
		\end{equation}
	\end{theorem}
    \begin{remark}
        With the help of error estimates in \eqref{proof of max u-u^h T-T^h error} and conditions in \eqref{time-condition-H1}, the restriction in \eqref{time-condition-H1} is available since 
        \begin{align*}
            \frac{\widehat{k}_n}{h} \| \nabla e_{n,\beta}^u \|^{2} 
            = \mathcal{O} (h^{2r+1} + k_{\max}), \quad 
            \widehat{k}_{n} \| \nabla e_{n,\beta}^u \|^{4} 
            \leq \frac{\big(\widehat{k}_{n} \| \nabla e_{n,\beta}^u \|^{2} \big)^{2}}{k_{\min}}
            = \mathcal{O} (k_{\max}).
        \end{align*}
    \end{remark}
	\begin{proof}
		By similar argument to \eqref{error equation 1}-\eqref{error equation 3}, we derive 
		\begin{align}
		\begin{split}\label{error equation 1 n alpha}
		&\frac{1}{\widehat{k}_n}(\psi^u_{n,\alpha}, v^h) 
		+ \mu (\nabla \psi^u_{n,\beta}, \nabla v^h) 
		+\gamma (\nabla \psi^w_{n,\beta},\nabla v^h)
		- (\psi^p_{n,\beta}, \nabla \cdot v^h)
		\\
		&=  \frac{1}{\widehat{k}_n}(\eta^u_{n,\alpha}, v^h) 
		+\nu b(u_{n,\beta}, u_{n,\beta}, v^h)
		-\nu b(u_{n,\beta}^h, u_{n,\beta}^h, v^h)
		- \rho (e^u_{n,\beta},v^h) 
		\\
		& \quad + \lambda (|u_{n,\beta}|^2 u_{n,\beta}  - |u_{n,\beta}^h|^2 u_{n,\beta}^h ,v^h)
		- \tau(u,w,p, v^h),
		\end{split}
		\\
		&(\frac{1}{\widehat{k}_n}\psi^w_{n,\alpha}, \varphi^h) - (\frac{1}{\widehat{k}_n} \nabla \psi^u_{n,\alpha}, \nabla \varphi^h)
		=0, \label{error equation 2 n alpha}
		\\
		&(\frac{1}{\widehat{k}_n} \nabla \cdot \psi^u_{n,\alpha} , q^h)=0.
		\label{error equation 3 n alpha}
		\end{align}
		We set \( v^h = \widehat{k}_n^{-1} \psi^u_{n,\alpha} \) in \eqref{error equation 1 n alpha}, \( \varphi^h = \psi^w_{n,\beta} \) in \eqref{error equation 2 n alpha}, 
		and  \( q^h = \psi^p_{n,\beta} \) in \eqref{error equation 3 n alpha} to derive 
		\begin{align}
			& \| \widehat{k}_n^{-1} \psi^u_{n,\alpha} \|^{2}   
			+ \frac{\mu}{ \widehat{k}_n } (\nabla \psi^u_{n,\beta}, \nabla \psi^u_{n,\alpha}) 
			+ \frac{\gamma}{\widehat{k}_n} (\psi^w_{n,\alpha}, \psi^w_{n,\beta})
			\label{error equation 1 n alpha-eq2} \\
			=& (\widehat{k}_n^{-1} \eta^u_{n,\alpha}, \widehat{k}_n^{-1} \psi^u_{n,\alpha}) 
			+\nu \big( b(u_{n,\beta}, u_{n,\beta}, \widehat{k}_n^{-1} \psi^u_{n,\alpha})
			- b(u_{n,\beta}^h, u_{n,\beta}^h, \widehat{k}_n^{-1} \psi^u_{n,\alpha}) \big) \notag \\
			&- \rho (e^u_{n,\beta},\widehat{k}_n^{-1} \psi^u_{n,\alpha}) 
			+ \lambda (|u_{n,\beta}|^2 u_{n,\beta}  - |u_{n,\beta}^h|^2 u_{n,\beta}^h ,\widehat{k}_n^{-1} \psi^u_{n,\alpha})
			- \tau(u,w,p, \widehat{k}_n^{-1} \psi^u_{n,\alpha}), \notag 
		\end{align}
		We multiply \eqref{error equation 1 n alpha-eq2} by $\widehat{k}_n$, sum the resulting equality over $n$ from $1$ to $m-1$ and utilize the identity in \eqref{G-stable}
		\begin{align}\label{error estimate nabla L2 inequality}
		&\frac{\mu}{4} (1+\theta) \|\nabla \psi^u_m\|^2 
		+ \frac{\mu}{4} (1-\theta)  \|\nabla \psi^u_{m-1}\|^2 
		+ \sum_{n=1}^{m-1} \widehat{k}_n \|\widehat{k}_n^{-1} \psi^u_{n,\alpha}\|^2
		\notag
		\\
		&+\frac{\gamma}{4} (1+\theta) \| \psi^w_m\|^2 + \frac{\gamma}{4} (1-\theta)  \| \psi^w_{m-1}\|^2 
		\notag
		\\
		=&			
		\sum_{n=1}^{m-1}\widehat{k}_n (\widehat{k}_n^{-1} \eta^u_{n,\alpha},
		\widehat{k}_n^{-1} \psi^u_{n,\alpha})
		+\sum_{n=1}^{m-1}	\widehat{k}_n \rho (e^u_{n,\beta}, \widehat{k}_n^{-1} \psi^u_{n,\alpha})
		\notag
		\\
		&+\sum_{n=1}^{m-1} \widehat{k}_n  \nu
		\left(
		b(u_{n,\beta}, u_{n,\beta},\widehat{k}_n^{-1} \psi^u_{n,\alpha})
		-b(u_{n,\beta}^h, u_{n,\beta}^h, \widehat{k}_n^{-1} \psi^u_{n,\alpha}) 
		\right) 
		\\
		&
		+ \sum_{n=1}^{m-1}	\widehat{k}_n \lambda (|u_{n,\beta}|^2 u_{n,\beta} - |u^h_{n,\beta}|^2u^h_{n,\beta}, \widehat{k}_n^{-1} \psi^u_{n,\alpha})
		-\sum_{n=1}^{m-1}\widehat{k}_n \tau(u,w,p, \widehat{k}_n^{-1} \psi^u_{n,\alpha})
		\notag
		\\
		&+ \frac{\mu}{4}(1+\theta)\|\nabla \psi^u_1\|^2 + \frac{\mu}{4}(1-\theta)\|\nabla \psi^u_{0}\|^2 
		+ \frac{\gamma}{4}(1+\theta)\|\psi^w_1\|^2 + \frac{\gamma}{4}(1-\theta)\|\psi^w_{0}\|^2. \notag 
		\end{align} 
        By similar argument to \eqref{derivation-term-eq1}-\eqref{derivation-term-eq3}
		\begin{align}
				&(\widehat{k}_n^{-1} \eta^u_{n,\alpha},
				\widehat{k}_n^{-1} \psi^u_{n,\alpha})
				\label{error-H1-term1} \\
				\leq&
				C(\theta) \frac{h^{2r+4}}{\widehat{k}_n} \int_{t_{n-1}}^{t_{n+1}} \big( \| u_t \|_{r+2}^2 + \| w_t \|_{r+2}^2 + \| p_t \|_{r+1}^2\big) \, \mathrm{d}t
				+ \frac{1}{16} \| \widehat{k}_n^{-1} \psi^u_{n,\alpha}\|^2. \notag 
		\end{align}
		By Cauchy-Schwarz inequality and Poincar\'e inequality 
		\begin{align}
			(e^u_{n,\beta}, \widehat{k}_n^{-1} \psi^u_{n,\alpha})
			\leq  C \| \nabla e_{n,\beta}^u \|^{2} + 	\frac{1}{16} \| \widehat{k}_n^{-1} \psi^u_{n,\alpha}\|^2.
            \label{error-H1-term2}
		\end{align}
        By algebraic calculation, we have
		\begin{align}
				&\nu b(u_{n,\beta}, u_{n,\beta},\widehat{k}_n^{-1} \psi^u_{n,\alpha})
				-\nu b(u_{n,\beta}^h, u_{n,\beta}^h, \widehat{k}_n^{-1} \psi^u_{n,\alpha}) 
				\label{b-terms-H1} \\
				=& - \nu b(e_{n,\beta}^u, u_{n,\beta}, \widehat{k}_n^{-1} \psi^u_{n,\alpha})
				- \nu b(u_{n,\beta}^h, e_{n,\beta}^u, \widehat{k}_n^{-1} \psi^u_{n,\alpha})
				\notag \\
                =& - \nu b(e_{n,\beta}^u, u_{n,\beta}, \widehat{k}_n^{-1} \psi^u_{n,\alpha})
                - \nu b(\psi_{n,\beta}^u, e_{n,\beta}^u, \widehat{k}_n^{-1} \psi^u_{n,\alpha})
				+ \nu b(\eta_{n,\beta}^u, e_{n,\beta}^u, \widehat{k}_n^{-1} \psi^u_{n,\alpha}) \notag \\
                &- \nu b(u_{n,\beta}, e_{n,\beta}^u, \widehat{k}_n^{-1} \psi^u_{n,\alpha}). \notag 
		\end{align}  
        By \eqref{b 2 1 0}, \eqref{b 1 2 0}, Poincar\'e inequality and Young's inequality 
        \begin{align}
            &- \nu b(e_{n,\beta}^u, u_{n,\beta}, \widehat{k}_n^{-1} \psi^u_{n,\alpha})
            - \nu b(u_{n,\beta}, e_{n,\beta}^u, \widehat{k}_n^{-1} \psi^u_{n,\alpha}) 
            \label{b-terms-H1-eq1} \\
            &\leq C \nu^{2} \| \nabla e_{n,\beta}^u \|^{2} \| u_{n,\beta} \|_{2}^{2}
            + \frac{1}{32} \| \widehat{k}_n^{-1} \psi^u_{n,\alpha}\|^2. \notag 
		\end{align}
		By \eqref{b 1 1 0 1}, inverse inequality in \eqref{inverse-ineq}, Poincar\'e inequality and Young's inequality 
        \begin{align}
            &- \nu b(\psi_{n,\beta}^u, e_{n,\beta}^u, \widehat{k}_n^{-1} \psi^u_{n,\alpha})
			+ \nu b(\eta_{n,\beta}^u, e_{n,\beta}^u, \widehat{k}_n^{-1} \psi^u_{n,\alpha})  
            \label{b-terms-H1-eq2} \\
            \leq& C \nu h^{-1/2}\big( \| \nabla \psi_{n,\beta}^u \| + \| \eta_{n,\beta}^u \|_{1} )
            \| \nabla e_{n,\beta}^u \| 
            \| \widehat{k}_n^{-1} \psi^u_{n,\alpha} \| \notag \\
            \leq& C \nu^{2} h^{-1}\big( \| \nabla \psi_{n,\beta}^u \|^{2} + \| \eta_{n,\beta}^u \|_{1}^{2} ) \| \nabla e_{n,\beta}^u \|^{2} 
            + \frac{1}{32} \| \widehat{k}_n^{-1} \psi^u_{n,\alpha}\|^2. \notag 
        \end{align}
		We combine \eqref{b-terms-H1} - \eqref{b-terms-H1-eq2} and use approximation 
        in \eqref{Stokes-type projection} to obtain
        \begin{align}
            &\nu b(u_{n,\beta}, u_{n,\beta},\widehat{k}_n^{-1} \psi^u_{n,\alpha})
			-\nu b(u_{n,\beta}^h, u_{n,\beta}^h, \widehat{k}_n^{-1} \psi^u_{n,\alpha}) 
			\label{b-terms-H1-conclusion} \\
            \leq& C \nu^{2} \| \nabla e_{n,\beta}^u \|^{2} \| u_{n,\beta} \|_{2}^{2}
            + \frac{C \nu^2}{h} \big( \| \nabla \psi_{n,\beta}^u \|^{2} + \| \eta_{n,\beta}^u \|_{1}^{2} ) \| \nabla e_{n,\beta}^u \|^{2} + \frac{1}{16} \| \widehat{k}_n^{-1} \psi^u_{n,\alpha}\|^2 \notag  \\
            \leq& C \nu^2 h^{-1} \| \nabla e_{n,\beta}^u \|^{2} \| \nabla \psi_{n,\beta}^u \|^{2} 
            + \frac{1}{16} \| \widehat{k}_n^{-1} \psi^u_{n,\alpha}\|^2 \notag  \\
            &+ C \nu^2 \Big[ \| u \|_{\infty,2}^{2} + h^{2r+1} \big( \| u \|_{\infty,r+2}^{2} 
            + \| w \|_{\infty,r+2}^{2} + \| p \|_{\infty,r+1}^{2} \big) \Big] \| \nabla e_{n,\beta}^u \|^{2}
            \notag 
		\end{align}
		By Cauchy-Schwarz inequality, continuity property in \eqref{vector_norm_inequality_2},
        Sobolev embedding inequality, Poincar\'e inequality and approximation 
        in \eqref{Stokes-type projection}
        \begin{align} 
            &\lambda |(|u_{n,\beta}|^2 u_{n,\beta} - | u_{n,\beta}^h|^2 u_{n,\beta}^h, \widehat{k}_n^{-1} \psi^u_{n,\alpha})| \notag \\
            \leq
            & 
            \lambda \Big(
            \int_{D} (|u_{n,\beta}|^2 u_{n,\beta} -|u_{n,\beta}^h|^2 u_{n,\beta}^h)^{2}\mathrm{d}x
            \Big)^\frac{1}{2} \| \widehat{k}_n^{-1} \psi^u_{n,\alpha} \|  \notag
            \\
            \leq& 
            C \lambda \Big( \int_{D} (|u_{n,\beta}| + |u_{n,\beta}^h| )^4 |u_{n,\beta} - u_{n,\beta}^h|^2 \mathrm{d}x
            \Big)^\frac{1}{2} \| \widehat{k}_n^{-1} \psi^u_{n,\alpha} \|  \notag \\
            \leq
            &
            C \lambda \Big(
            \int_{D} (|u_{n,\beta}| + |u_{n,\beta}^h| )^8 \mathrm{d}x
            \Big)^\frac{1}{4}
            \Big(
            \int_{D} |u_{n,\beta} - u_{n,\beta}^h|^4 \mathrm{d}x
            \Big)^\frac{1}{4} 
            \| \widehat{k}_n^{-1} \psi^u_{n,\alpha} \|
            \label{cubic-term-H1}    \\
            \leq
            & C \lambda 
            \Big( \| u_{n,\beta} \|_{L^8(D)}^2 + \| e_{n,\beta}^u \|_{L^8(D)}^2 \Big) 
            \|\nabla e_{n,\beta}^u \| \| \widehat{k}_n^{-1} \psi^u_{n,\alpha} \| \notag
            \\ 
            \leq
            & C \lambda (\|u_{n,\beta}\|_{1}^2 + \| \nabla e_{n,\beta} \|_{1}^2 )
            \big( \|\nabla \psi_{n,\beta}^{u} \| + \|\nabla \eta_{n,\beta}^{u} \| \big)
            \| \widehat{k}_n^{-1} \psi^u_{n,\alpha} \| \notag \\
            \leq& C \lambda^2 (\|u_{n,\beta}\|_{1}^4 + \| \nabla e_{n,\beta} \|_{1}^4 )
            \big( \|\nabla \psi_{n,\beta}^{u} \|^{2} + \|\nabla \eta_{n,\beta}^{u} \|^{2} \big)
            + \frac{1}{16} \| \widehat{k}_n^{-1} \psi^u_{n,\alpha}\|^2. \notag \\
            \leq& C \lambda^2 (\|u_{n,\beta}\|_{1}^4 + \| \nabla e_{n,\beta} \|_{1}^4 ) 
            \|\nabla \psi_{n,\beta}^{u} \|^{2} + \frac{1}{16} \| \widehat{k}_n^{-1} \psi^u_{n,\alpha}\|^2 \notag \\
            &+ C \lambda^2 h^{2r+2} (\|u_{n,\beta}\|_{1}^4 + \| \nabla e_{n,\beta} \|_{1}^4 ) 
            \big( \| u \|_{\infty,r+2}^{2} 
            + \| w \|_{\infty,r+2}^{2} + \| p \|_{\infty,r+1}^{2} \big)
            \notag 
        \end{align}
        Now we address terms in $\tau(u,w,p,\widehat{k}_n^{-1} \psi^u_{n,\alpha})$. 
        By similar argument to \eqref{cubic-term-H1}
        \begin{align} \label{tau-H1-eq1} 
            \begin{split}
            &\lambda \big| \big( |u_{n,\beta}|^2 u_{n,\beta} - | u(t_{n,\beta}) |^2 u(t_{n,\beta}), \widehat{k}_n^{-1} \psi^u_{n,\alpha} \big) \big|  \\
            \leq& C \lambda^2 (\|u_{n,\beta}\|_{1}^4 + \| u(t_{n,\beta}) \|_{1}^4 )
            \| u_{n,\beta} - u(t_{n,\beta}) \|_{1}^{2}
            + \frac{1}{32} \| \widehat{k}_n^{-1} \psi^u_{n,\alpha}\|^2 \\
            \leq& C(\theta) \lambda^2 \| u \|_{\infty,1}^{4} (k_{n} + k_{n-1})^{3} 
            \int_{t_{n-1}}^{t_{n+1}} \| u_{tt} \|_{1}^{2} dt 
            + \frac{1}{32} \| \widehat{k}_n^{-1} \psi^u_{n,\alpha}\|^2.
            \end{split}
        \end{align}
        By \eqref{b 1 2 0}, \eqref{b 2 1 0}, Young's inequality and Lemma \ref{n beta int tn-1 tn+1}
        \begin{align} \label{tau-H1-eq2}
			\begin{split}
				&\nu \big(b(u_{n,\beta}, u_{n,\beta}, \widehat{k}_n^{-1} \psi^u_{n,\alpha})
                - b(u(t_{n,\beta}), u(t_{n,\beta}), \widehat{k}_n^{-1} \psi^u_{n,\alpha}) \big)
				\\
				\leq&
				C\nu (\| u_{n,\beta} \|_{2} + \| u(t_{n,\beta}) \|_{2} )  \| u_{n,\beta} - u(t_{n,\beta})\|_{1} \| \widehat{k}_n^{-1} \psi^u_{n,\alpha} \| 		
				\\ 
				\leq&
                C(\theta) \nu^{2} \| u \|_{\infty,2}^{2} (k_{n} + k_{n-1})^{3}  
                \int_{t_{n-1}}^{t_{n+1}} \| u_{tt} \|_{1}^{2} dt
                + \frac{1}{32} \| \widehat{k}_n^{-1} \psi^u_{n,\alpha}\|^2. 
			\end{split}
		\end{align}
        For the remaining truncation terms in \eqref{truncation}, 
        we utilize Cauchy-Schwarz, inequality, Gauss divergence theorem, Poincar\'e inequality, Young's inequalities, along with Lemma \ref{n beta int tn-1 tn+1}
        \begin{align}
			&\Big( \frac{u_{n,\alpha}}{\widehat{k}_n} - u_t(t_{n,\beta}), \widehat{k}_n^{-1} \psi^u_{n,\alpha} \Big) 
			\leq
			C(\theta) k_{\max}^3 \int_{t_{n-1}}^{t_{n+1}} \|u_{ttt}\|^2  \mathrm{d}t
			+ \frac{1}{32} \| \widehat{k}_n^{-1} \psi^u_{n,\alpha} \|^2,
			\label{tau-H1-eq3}  \\
			&\mu (\nabla (u_{n,\beta} -u(t_{n,\beta})), \nabla \widehat{k}_n^{-1} \psi^u_{n,\alpha} ) 
			\leq 
			C(\theta) \mu^{2} k_{\max}^3 \int_{t_{n-1}}^{t_{n+1}}\|\Delta u_{tt}\|^2  \mathrm{d}t
			+\frac{1}{32}\| \widehat{k}_n^{-1} \psi^u_{n,\alpha} \|^2,
			\notag \\
			&\rho (u_{n,\beta}-u(t_{n,\beta}), \widehat{k}_n^{-1} \psi^u_{n,\alpha})
			\leq 
			C(\theta) \rho^2 k_{\max}^3 \int_{t_{n-1}}^{t_{n+1}}\|u_{tt}\|^2  \mathrm{d}t
			+ \frac{1}{32} \| \widehat{k}_n^{-1} \psi^u_{n,\alpha}\|^2,
			\notag \\
			&\gamma(\nabla (w_{n,\beta} - w(t_{n,\beta})), \nabla \widehat{k}_n^{-1} \psi^u_{n,\alpha})
			\leq  
			C(\theta) \gamma^2 k_{\max}^3 \int_{t_{n-1}}^{t_{n+1}}\|\Delta w_{tt}\|^2  \mathrm{d}t
			+\frac{1}{32}  \| \widehat{k}_n^{-1} \psi^u_{n,\alpha} \|^2,
			\notag \\
			&(p_{n,\beta} - p(t_{n,\beta}), \nabla \cdot \widehat{k}_n^{-1} \psi^u_{n,\alpha})
			\leq  
			C(\theta) k_{\max}^3 \int_{t_{n-1}}^{t_{n+1}}\|\nabla p_{tt}\|^2  \mathrm{d}t
			+\frac{1}{32} \|\widehat{k}_n^{-1} \psi^u_{n,\alpha} \|^2,
			\notag \\
			&(f_{n,\beta} - f(t_{n,\beta}),\widehat{k}_n^{-1} \psi^u_{n,\alpha})
			\leq   
			C(\theta) k_{\max}^3 \int_{t_{n-1}}^{t_{n+1}}\|f_{tt}\|^2  \mathrm{d}t
			+\frac{1}{32} \|\widehat{k}_n^{-1} \psi^u_{n,\alpha}\|^2. \notag 
		\end{align}
        We combine \eqref{tau-H1-eq1} - \eqref{tau-H1-eq3} to have 
		\begin{align} 
			&\sum_{n=1}^{m-1} \widehat{k}_n  |\tau(u,w,p,\widehat{k}_n^{-1} \psi^u_{n,\alpha})|
			\label{tau-H1-conclusion} \\
			&\leq C(\theta) \big( \lambda^2 \| u \|_{\infty,1}^{4} + \nu^{2} \| u \|_{\infty,2}^{2} \big) k_{\max}^4 \| u_{tt} \|_{2,1}^{2}
            + \frac{1}{4} \|\widehat{k}_n^{-1} \psi^u_{n,\alpha}\|^2 \notag \\
            &+ \!C(\theta) k_{\max}^4 \!\! \int_{t_{n-1}}^{t_{n+1}} \!\!\big( \|u_{ttt}\|^2 
            + \mu^{2} \|\Delta u_{tt}\|^2 + \rho^2 \|u_{tt}\|^2 + \gamma^2 \|\Delta w_{tt}\|^2
            + \|\nabla p_{tt}\|^2 + \|f_{tt}\|^2 \big) \mathrm{d}t. \notag 
		\end{align}
        By \eqref{error-H1-term1}, \eqref{error-H1-term2}, \eqref{b-terms-H1-conclusion}, 
        \eqref{cubic-term-H1} and \eqref{tau-H1-conclusion}, 
        \eqref{error estimate nabla L2 inequality} becomes 
        \begin{align}
            &\frac{\mu}{4} (1+\theta) \|\nabla \psi^u_m\|^2 
            +\frac{\gamma}{4} (1+\theta) \| \psi^w_m\|^2
            + \sum_{n=1}^{m-1} \frac{\widehat{k}_n}{2} \|\widehat{k}_n^{-1} \psi^u_{n,\alpha}\|^2
            \label{error estimate H1 conclusion} \\
            &= C(\theta) \Big[ \frac{\nu^2 }{h} \| \nabla e_{m-1,\beta}^u \|^{2} 
            + \lambda^2 \big( \|u_{m-1,\beta}\|_{1}^4 + \| \nabla e_{m-1,\beta} \|_{1}^4 \big) \Big]
            \widehat{k}_{m-1} \| \nabla \psi_{m}^u \|^2  \notag \\
            &+ C(\theta) \Big[ \frac{\nu^2}{h} \| \nabla e_{m-1,\beta}^u \|^{2} 
            \!+\! \lambda^2 \big( \|u_{m-1,\beta}\|_{1}^4 + \| \nabla e_{m-1,\beta} \|_{1}^4 \big) \Big] \widehat{k}_{m-1}
            \big( \| \nabla \psi_{m-1}^u \|^2 \!+\! \| \nabla \psi_{m-2}^u \|^2 \big) \notag \\
            &+C \sum_{n=1}^{m-2} \Big[ \frac{\nu^2 }{h} \| \nabla e_{n,\beta}^u \|^{2} 
            + \lambda^2 \big( \|u_{n,\beta}\|_{1}^4 + \| \nabla e_{n,\beta} \|_{1}^4 \big) \Big]
            \widehat{k}_{n} \| \nabla \psi_{n,\beta}^u \|^2		
            + C \sum_{n=1}^{m-1} \widehat{k}_{n} \|e_{n,\beta}^{u} \|^{2}  \notag \\
            &+ C \nu^2 \Big[ \| u \|_{\infty,2}^{2} + h^{2r+1} \big( \| u \|_{\infty,r+2}^{2} 
            + \| w \|_{\infty,r+2}^{2} + \| p \|_{\infty,r+1}^{2} \big) \Big] 
            \sum_{n=1}^{m-1} \widehat{k}_{n} \|e_{n,\beta}^{u} \|^{2} \notag \\
            &+ C \lambda^2 h^{2r+2} \big( \| u \|_{\infty,r+2}^{2} + \| w \|_{\infty,r+2}^{2} + \| p \|_{\infty,r+1}^{2} \big) 
            \sum_{n=1}^{m-1} \widehat{k}_{n} (\|u_{n,\beta}\|_{1}^4 
            + \| \nabla e_{n,\beta} \|_{1}^4 )  \notag  \\
            &+\!C(\theta) h^{2r+4} \big(\! \| u_t \|_{2,r\!+\!2}^2 \!+\! \| w_t \|_{2,r\!+\!2}^2 
            \!+\! \| p_t \|_{2,r\!+\!1}^2 \! \big)
            \!+\! C(\theta) \big(\! \lambda^2 \| u \|_{\infty,1}^{4} \!+\! \nu^{2} \| u \|_{\infty,2}^{2} \! \big) k_{\max}^4 \| u_{tt} \|_{2,1}^{2} \notag \\
            &+ C(\theta) k_{\max}^4 \big( \|u_{ttt}\|_{2,0}^2 + \mu^{2} \|\Delta u_{tt}\|_{2,0}^2 + \rho^2 \|u_{tt}\|_{2,0}^2 + \gamma^2 \|\Delta w_{tt}\|_{2,0}^2 
            + \|f_{tt}\|_{2,0}^2 \big). \notag 
        \end{align} 
        We make use of discrete Gr\"onwall inequality \cite[p.369]{MR1043610} along with restrictions in \eqref{time-condition-H1} to obtain
        \begin{align}
            \label{psi H1 conclusion}
            &\|\nabla \psi^u_m\|^2 
            +\| \psi^w_m\|^2
            + \sum_{n=1}^{m-1} \frac{\widehat{k}_n}{2} \|\widehat{k}_n^{-1} \psi^u_{n,\alpha}\|^2
            \leq \mathcal{O} (h^{2r+2}+ k_{\max}^{4}),
        \end{align}
		which implies \eqref{proof of max nabla u-u^h T-T^h error} by triangle inequality and approximation in \eqref{Stokes-type projection}.
	\end{proof}

	\subsection{Error estimate for pressure} \

    \begin{lemma} \label{Intermediate conclusion 3}
		Suppose the assumption of \ref{u_p_T_space}, \ref{time_condition_2} and 
        conditions \eqref{time-diameter-condition} hold. 
        If the time step size and uniform mesh diameter are sufficiently small such that 
        the restriction in \eqref{time-condition-H1} for all $1 \leq n \leq M-1$, 
        then Scheme \ref{fully discrete formulations scheme second} has the following error estimates 
		\begin{align}
            \label{lemma-p-conclusion}
            &\sum_{n=1}^{m-1}  \mu \widehat{k}_n \|\psi^w_{n,\beta}\|^2 
            + \frac{\gamma}{2} \sum_{n=1}^{m-1} \widehat{k}_n \|\nabla \psi^w_{n,\beta}\|^2 
            \leq \mathcal{O} (h^{2r+2} + k_{\max}^{4}).
        \end{align}
	\end{lemma}
    \begin{proof}
        We set \( v^h = \psi^w_{n,\beta} \) in \eqref{error equation 1} and \( \varphi^h = \psi^w_{n,\beta} \) in \eqref{error equation 2} and use the fact $\psi^w_{n,\beta} \in V_h$ to derive 
        \begin{align} 
                &\mu \| \psi^w_{n,\beta} \|^{2}
                +\gamma \| \nabla \psi^w_{n,\beta} \|^{2}
                \label{lemma-p-eq1} \\
                &=  - \frac{1}{\widehat{k}_n}(e^u_{n,\alpha}, \psi^w_{n,\beta}) 
                +\nu b(u_{n,\beta}, u_{n,\beta}, \psi^w_{n,\beta})
                -\nu b(u_{n,\beta}^h, u_{n,\beta}^h, \psi^w_{n,\beta})
                - \rho (e^u_{n,\beta},\psi^w_{n,\beta}) 
                \notag \\
                & \quad + \lambda (|u_{n,\beta}|^2 u_{n,\beta}  - |u_{n,\beta}^h|^2 u_{n,\beta}^h ,v^h)
                - \tau(u,w,p, \psi^w_{n,\beta}). \notag 
        \end{align}
        We multiply \eqref{lemma-p-eq1} by $\widehat{k}_n$ and sum the resulting equality over $n$ from $1$ to $m-1$
        \begin{align} 
				&\sum_{n=1}^{m-1}  \mu \widehat{k}_n \|\psi^w_{n,\beta}\|^2 
				+ \sum_{n=1}^{m-1}  \gamma \widehat{k}_n \|\nabla \psi^w_{n,\beta}\|^2 
				\label{zeq:0601-4.40} \\
				=&
				-\sum_{n=1}^{m-1} \widehat{k}_n \big[ (\widehat{k}_n^{-1} e^u_{n,\alpha},\psi^w_{n,\beta})
				+ \rho (e^u_{n,\beta},\psi^w_{n,\beta}) \big]
                +\sum_{n=1}^{m-1} \widehat{k}_n \lambda 
				(|u_{n,\beta}|^2 u_{n,\beta} - |u^h_{n,\beta}|^2u^h_{n,\beta},\psi^w_{n,\beta})
				\notag \\
				&+\sum_{n=1}^{m-1} \nu\widehat{k}_n( b(u_{n,\beta},u_{n,\beta},\psi^w_{n,\beta})
				- b(u^h_{n,\beta},u^h_{n,\beta},\psi^w_{n,\beta}))
				- \sum_{n=1}^{m-1} \widehat{k}_n \tau(u,w,p,\psi^w_{n,\beta}). \notag 
		\end{align}
        By Cauchy-Schwarz inequality, Poincar\'e inequality and Young's inequality 
		\begin{align} \label{lemma-p-eq3}
			\begin{split}
				&\sum_{n=1}^{m-1} \widehat{k}_n (\widehat{k}_n^{-1} e^u_{n,\alpha},\psi^w_{n,\beta})
				+ \sum_{n=1}^{m-1}\widehat{k}_n \rho (e^u_{n,\beta},\psi^w_{n,\beta}) \\
				\leq& \frac{C}{\gamma} \sum_{n=1}^{m-1} \widehat{k}_n \| \widehat{k}_n^{-1} e^u_{n,\alpha} \|^{2} 
                + \frac{C |\rho|}{\gamma} \sum_{n=1}^{m-1} \widehat{k}_n \| \nabla e^u_{n,\beta} \|
				+ \frac{\gamma}{8} \sum_{n=1}^{m-1} \widehat{k}_n  \| \nabla \psi^w_{n,\beta}\|^2 
			\end{split}
		\end{align}
        By similar argument to \eqref{cubic term inequality 1}
        \begin{align} \label{lemma-p-eq4}
            \begin{split}
			&\sum_{n=1}^{m-1}	\widehat{k}_n \lambda (|u_{n,\beta}|^2 u_{n,\beta} - |u^h_{n,\beta}|^2u^h_{n,\beta}, \psi^w_{n,\beta}) 
			\\
            \leq& C \lambda \sum_{n=1}^{m-1} \widehat{k}_n \big( \| u_{n,\beta} \|_{L^{4}}^{2} 
            + \| e_{n,\beta}^{u} \|_{1}^{2} \big) \| \nabla e_{n,\beta}^{u} \| 
            \| \nabla \psi^w_{n,\beta} \|
            \\ 
			\leq& \frac{C \lambda^2}{\gamma} \sum_{n=1}^{m-1}	\widehat{k}_n 
			\big( \|u_{n,\beta}\|_{L^4}^4 + \| e_{n,\beta}^{u} \|_{1}^{4} \big)
			\| \nabla e_{n,\beta}^u\|^2
			+\frac{\gamma}{8} \sum_{n=1}^{m-1} \widehat{k}_n  \|\nabla \psi^w_{n,\beta}\|^2 \\
			\leq& \frac{C(\theta) \lambda^2}{\gamma} \big( \| u \|_{L^{\infty}(L^4)}^{4} 
            + \| e^u \|_{\infty,1}^{4} \big)
			\sum_{n=1}^{m-1}	\widehat{k}_n \| \nabla e_{n,\beta}^u\|^2
			+ \frac{\gamma}{8} \sum_{n=1}^{m-1} \widehat{k}_n  \|\nabla \psi^w_{n,\beta}\|^2.
            \end{split}
		\end{align}
		By \eqref{b 1 1 1} and Young's inequality
		\begin{align}
			&\sum_{n=1}^{m-1} \nu\widehat{k}_n( b(u_{n,\beta},u_{n,\beta},\psi^w_{n,\beta})
			- b(u^h_{n,\beta},u^h_{n,\beta},\psi^w_{n,\beta}))
			\label{lemma-p-eq5} \\
			=& \sum_{n=1}^{m-1}  \nu \widehat{k}_n \big( - b(e_{n,\beta}^{u},u_{n,\beta},\psi^w_{n,\beta}) 
			- b(e_{n,\beta}^u,e_{n,\beta}^{u},\psi^w_{n,\beta}) 
			- b(u_{n,\beta},e_{n,\beta}^{u},\psi^w_{n,\beta})  \big) 
			\notag \\
			\leq&
			\sum_{n=1}^{m-1}C \widehat{k}_n  \nu(\|\nabla u_{n,\beta}\| + \|\nabla e_{n,\beta}^u \|)\|\nabla e^u_{n,\beta}\|\|\nabla \psi^w_{n,\beta}\| 
			\notag \\
			\leq&
			\frac{C(\theta) \nu^2}{\gamma} \big( \| \nabla u \|_{\infty,0}^2 + \| \nabla e^{u} \|_{\infty,0}^2 \big) \sum_{n=1}^{m-1} \widehat{k}_n \|  \nabla e_{n,\beta}^u\|^2
			+\frac{\gamma}{8} \sum_{n=1}^{m-1} \widehat{k}_n  \|\nabla \psi^w_{n,\beta}\|^2. 
			\notag
		\end{align}
		By similar argument to \eqref{truncation error estimate equation}
		\begin{align} 
		&\sum_{n=1}^{m-1} \widehat{k}_n  |\tau(u,w,p,\psi^w_{n,\beta})|
		\label{lemma-p-eq6} \\
		\leq&
		\frac{\gamma}{8} \sum_{n=1}^{m-1} \widehat{k}_n  \| \nabla \psi^w_{n,\beta}\|^2
		+ \frac{C(\theta)}{\gamma} \big( \nu^2 \| u \|_{\infty,2} + \lambda^2 \| u \|_{L^{\infty}(L^{4})}^{4} \big) k_{\rm{max}}^4 \| \nabla u_{tt} \|_{2,0}^{2}
        \notag \\
		+& \frac{C(\theta) k_{\rm{max}}^4}{\gamma} \big( \|u_{ttt}\|_{2,0}^2 
        + \mu^2 \|\nabla u_{tt}\|_{2,0}^2 + \rho^2 \|u_{tt}\|_{2,0}^2 
        + \gamma^{2} \|\nabla w_{tt}\|_{2,0}^2 + \| p_{tt}\|_{2,0}^2 + \|f_{tt}\|_{2,0}^2 \big). \notag 
		\end{align}
        By \eqref{zeq:0601-4.40} - \eqref{lemma-p-eq6} and error estimates 
        in \eqref{derivation-term-conclusion}, \eqref{proof of max nabla u-u^h T-T^h error}, 
        we have \eqref{lemma-p-conclusion}. 
    \end{proof}

    \begin{theorem}\label{Intermediate conclusion 4}
        Suppose the assumption of \ref{u_p_T_space}, \ref{time_condition_2} and 
        conditions \eqref{time-diameter-condition} hold. 
        If the time step size and uniform mesh diameter are sufficiently small such that 
        the restriction in \eqref{time-condition-H1} for all $1 \leq n \leq M-1$, 
        then Scheme \ref{fully discrete formulations scheme second} has the following error estimate for pressure 
        \begin{equation} \label{nabla p-p^h error}
            \sum_{n=1}^{M-1} \widehat{k}_n
            (\|p(t_{n,\beta})- p_{n,\beta}^h\|^2)
            \leq \mathcal{O}(h^{2r+2} + k_{\max}^4).
        \end{equation}
    \end{theorem}
    \begin{proof}
        By \eqref{error equation 1}, definition of Stokes-type projection 
        $\mathcal{S}_h p_{n,\beta}$ in \eqref{Stokes-type-projection-defination},
        Poincar\'e inequality and similar argument to \eqref{lemma-p-eq4} - \eqref{lemma-p-eq6} 
        \begin{align}
                & \big( \psi_{n,\beta}^p, \nabla \cdot v^h \big)  \label{p - Shp n beta} \\
                =&
                \big({\widehat{k}_n}^{-1} e^u_{n,\alpha}, v^h \big) 
                +\mu (\nabla e^u_{n,\beta}, \nabla v^h) 
                +\gamma (\nabla e^w_{n,\beta}, \nabla v^h)
                +\rho(e^u_{n,\beta},v^h) + (\eta_{n,\beta}^p, \nabla \cdot v^h)
                \notag \\
                &+\nu b(u_{n,\beta}^h, u_{n,\beta}^h, v^h) - \nu b(u_{n,\beta}, u_{n,\beta}, v^h)
                -\lambda(|u_{n,\beta}|^2 u_{n,\beta} - |u^h_{n,\beta}|^2u^h_{n,\beta},v^h)
                \notag \\
                &+ \tau(u, w,p, v^h) \notag  \\
                \leq& C(\theta) \| \nabla v^h \| \Big\{ \| {\widehat{k}_n}^{-1} e^u_{n,\alpha} \| 
                + \mu \| \nabla e^u_{n,\beta} \| 
                + \gamma \| \nabla e^w_{n,\beta} \| 
                + |\rho| \| \nabla e^u_{n,\beta} \|  
                + C \| \eta_{n,\beta}^p \| \notag \\
                &+ \nu \big( \| \nabla u \|_{\infty,0} + \| \nabla e^{u} \|_{\infty,0} \big) 
                \|\nabla e^u_{n,\beta}\|  
                + \lambda \big( \| u_{n,\beta} \|_{L^{\infty}(L^4)}^{2}
                + \| e_{n,\beta}^{u} \|_{\infty,1}^{2} \big) \| \nabla e_{n,\beta}^{u} \| 
                \notag \\
                &+ \big( \nu \| u \|_{\infty,2} + \lambda \| u \|_{L^{\infty}(L^{4})}^{2} \big)
                \Big( k_{\rm{max}}^3 \int_{t_{n-1}}^{t_{n+1}}\|\nabla u_{tt}\|^2  \mathrm{d}t \Big)^{\frac{1}{2}}   \notag \\
				&+ \Big[ \!k_{\rm{max}}^3 \!\!\! \int_{t_{n-1}}^{t_{n+1}} \!\!\!\big(
				\|u_{ttt}\|^2 \!+\! \mu^2 \|\nabla u_{tt}\|^2 \!+\! \rho^2 \|u_{tt}\|^2 \!+\! \gamma^{2} \|\nabla w_{tt}\|^2 \!+\! \| p_{tt}\|^2 \!+\! \|f_{tt}\|^2
				\big)	 \mathrm{d}t \! \Big]^{\frac{1}{2}} \Big\}. \notag 
        \end{align}
        By the $LBB^h$ condition in \eqref{LBB}, 
        \begin{align}
            \label{p-LBB}
            \|\psi_{n,\beta}^p\| \leq C \sup_{v^h \in X_h \setminus \{0\}} \frac{(\nabla \cdot v^h, \psi_{n,\beta}^p )}{\|\nabla v^h\|}.
        \end{align}
        Hence we combine \eqref{p - Shp n beta} and \eqref{p-LBB} to have 
        \begin{align}
            &\sum_{n=1}^{M-1}\widehat{k}_n \| \psi_{n,\beta}^p \|^{2}  \\
            &\leq C(\theta) \sum_{n=1}^{M-1}\widehat{k}_n 
            \big( \| {\widehat{k}_n}^{-1} e^u_{n,\alpha} \|^{2} 
            + \mu^{2} \| \nabla e^u_{n,\beta} \|^{2} 
            + \gamma^{2} \| \nabla e^w_{n,\beta} \|^{2} + \rho^{2} \| \nabla e^u_{n,\beta} \|^{2}  \big)  \notag \\
            &+ C(\theta) T h^{2r+2} \big( \| u \|_{\infty,r+2} + \| w \|_{\infty,r+2} 
            + \| p \|_{\infty,r+1} \big)^{2} \notag \\
            &+ C(\theta) \Big[\nu^2 \big( \| \nabla u \|_{\infty,0}^2 + \| \nabla e^{u} \|_{\infty,0}^2 \big)
            + \lambda^2 \big( \| u_{n,\beta} \|_{L^{\infty}(L^4)}^{4}
            + \| e_{n,\beta}^{u} \|_{\infty,1}^{4} \big) \Big]
            \sum_{n=1}^{M-1}\widehat{k}_n \|\nabla e^u_{n,\beta}\|^2 \notag \\
            &+ C(\theta)\big( \nu^2 \| u \|_{\infty,2}^2 + \lambda^2 \| u \|_{L^{\infty}(L^{4})}^{4} \big)
            k_{\rm{max}}^4 \|\nabla u_{tt}\|_{2,0}^2 \notag \\
            &+ C(\theta) k_{\rm{max}}^4 \big(  \|u_{ttt}\|_{2,0}^2 
            \!+\! \mu^2 \|\nabla u_{tt}\|_{2,0}^2 \!+\! \rho^2 \|u_{tt}\|_{2,0}^2 
            \!+\! \gamma^{2} \|\nabla w_{tt}\|_{2,0}^2 \!+\! \| p_{tt}\|_{2,0}^2 
            \!+\! \|f_{tt}\|_{2,0}^2  \big). \notag 
        \end{align}
        By error estimates in \eqref{proof of max u-u^h T-T^h error}, \eqref{proof of max nabla u-u^h T-T^h error}, \eqref{lemma-p-conclusion},  triangle inequality, approximation of 
        $\mathcal{S}_h p_{n,\beta}$ in \eqref{Stokes-type projection} 
        and Lemma \ref{n beta int tn-1 tn+1}, we achieve \eqref{nabla p-p^h error}.
    \end{proof}

    \section{Numerical results}  
    \label{sec:sec5}
    In this section, we present several numerical experiments to validate the theoretical analysis and assess the performance of the equivalent divergence-free DLN 
    scheme (Scheme \ref{fully discrete formulations scheme}) with the parameter $\theta = 0.3$.
    We utilize Taylor-Hood $\mathbb{P}2/\mathbb{P}1$ finite element space for spatial discretization among all the experiments. 
    We first perform a convergence test for the Stokes-type projection $\mathcal{S}_h$
    in \eqref{Stokes-type-projection-defination}, followed by another experiment to validate both spatial and temporal convergence rate of Scheme \ref{fully discrete formulations scheme}. 
    Then we investigate the self-organization dynamics of active fluid in a two-dimensional domain \cite{MR4736040} with random initial conditions to evaluate the robustness of 
    Scheme \ref{fully discrete formulations scheme} and efficiency of the time adaptive strategy based on the minimal dissipation criterion. 

    \subsection{Convergence test of Stokes type projection}
    To confirm that $\mathcal{S}_h$ satisfies the error estimate stated in \eqref{Stokes-type projection}, we implement the equivalent algorithm of $\mathcal{S}_h$
    in \eqref{Stokes-type-projection-defination 2} with parameters $\mu = \gamma = 1$, and construct the test problem having the following exact solutions on the unit square domain $D = [0,1]^2$
    \begin{align*}
        &\begin{bmatrix}
            u_{1} \\ u_{2}
        \end{bmatrix} = 
        \begin{bmatrix}
            ( -\cos(2\pi x+\pi)-1)\sin(2\pi y) \\
            -\sin(2\pi x) \cos(2\pi y)
        \end{bmatrix},  \\
        &\begin{bmatrix}
            w_{1} \\ w_{2}
        \end{bmatrix} = 
        \begin{bmatrix}
            -3x^2 + 3y^2 + 8\pi^2 \sin(2\pi y) \cos(2\pi x) - 4\pi^2 \sin(2\pi y) \\
            6xy - 8\pi^2\sin(2\pi x)\cos(2\pi y)
        \end{bmatrix},  \\
        &\ \ \phi = x^3-3xy^2, \quad  p = -\cos(2\pi x) - \cos(2 \pi y).
    \end{align*}
    We set mesh diameter $h = \frac{1}{128}, \frac{1}{256}, \frac{1}{512}, \frac{1}{1024}$ and obtain the corresponding errors in \( L^2 \) and  \( H^1 \)-norm, along with the convergence rates in \Cref{Stokes type $L^2$-errors and convergence rates in space,Stokes type $H^1$-errors and convergence rates in space}.
    The numerical results confirm the error estimate in \eqref{Stokes-type projection}:
    the velocity $u$ and the auxiliary variable $w$ achieve third-order accuracy in the $L^2$-norm,
    and second-order accuracy in $H^1$-norm while the pressure $p$ has second-order accuracy in $L^2$-norm.
    \begin{table}[h]
        \centering
        \caption{\(L^2\)-errors and convergence rates of the Stokes-type system in space}
        \label{Stokes type $L^2$-errors and convergence rates in space}
        \begin{tabular}{@{}lllllllll@{}}
            \hline
            \(1 / h\) & \( \|u - u^h\| \) & Rate & \( \|w - w^h\| \) & Rate & \( \|\phi - \phi^h\| \) & Rate & \( \|p - p^h\| \) & Rate \\
            \hline
            128  & 1.28E-5 & —      & 9.72E-4 & —      & 3.60E-4 & —      & 9.47E-4 & — \\
            256  & 1.60E-6 & 3.0045 & 1.22E-4 & 2.9992 & 8.98E-5 & 2.0018 & 1.18E-4 & 3.0095 \\
            512  & 1.99E-7 & 3.0011 & 1.52E-5 & 2.9998 & 2.25E-5 & 2.0004 & 2.34E-5 & 2.3286 \\
            1024 & 2.49E-8 & 3.0003 & 1.90E-6 & 3.0000 & 5.61E-6 & 2.0001 & 5.64E-6 & 2.0528 \\
            \hline
        \end{tabular}
    \end{table}

    \begin{table}[h]
        \centering
        \caption{\(H^1\)-errors and convergence rates of the Stokes-type system in space}
        \label{tab:stokes_H1_errors}
        \begin{tabular}{@{}lllllllll@{}}
            \hline
            \(1 / h\) & \( \|u - u^h\|_1 \) & Rate & \( \|w - w^h\|_1 \) & Rate & \( \|\phi - \phi^h\|_1 \) & Rate & \( \|p - p^h\|_1 \) & Rate \\
            \hline
            128  & 6.18E-3 & —      & 4.75E-1 & —      & 1.78E-1 & —      & 9.47E-4 & —      \\
            256  & 1.54E-3 & 1.9992 & 1.19E-1 & 1.9992 & 8.90E-2 & 1.0003 & 9.20E-2 & 1.2700 \\
            512  & 3.86E-4 & 1.9998 & 2.97E-2 & 1.9998 & 4.45E-2 & 1.0001 & 4.47E-2 & 1.0418 \\
            1024 & 9.66E-5 & 2.0000 & 7.43E-3 & 1.9999 & 2.23E-2 & 1.0000 & 2.23E-2 & 1.0054 \\
            \hline
        \end{tabular}
        \label{Stokes type $H^1$-errors and convergence rates in space}
    \end{table}

    \subsection{Convergence test of fully-discrete DLN scheme}
    To validate the convergence rate of Scheme \ref{fully discrete formulations scheme} 
    in both space and time, we construct the test problem on the unit square domain \( D = [0,1] \times [0,1] \) with the following exact solution
    \begin{align*}
        &\begin{bmatrix}
            u_{1} \\ u_{2}
        \end{bmatrix} = 
        \begin{bmatrix}
            ( -\cos(2\pi x+\pi)-1)\sin(2 \pi y) \exp(2t) \\
            -\sin(2\pi x) \cos(2\pi y) \exp(2t)
        \end{bmatrix},  \\
        &\begin{bmatrix}
            w_{1} \\ w_{2}
        \end{bmatrix} = 
        \begin{bmatrix}
            (-3x^2 + 3y^2 + 8\pi^2 \sin(2\pi y) \cos(2\pi x) - 4\pi^2 \sin(2\pi y)) \exp(2t) \\
            (6xy - 8\pi^2\sin(2\pi x)\cos(2\pi y))\exp(2t)
        \end{bmatrix},  \\
        &\ \ \phi = (x^3-3xy^2)\exp(2t), \quad  p = \sin(3\pi^2x)\cos(3\pi^2 y)\exp(-t).
    \end{align*}
    We set the physical parameters to \( \mu = 1 \), \( \gamma = 1 \), \( \nu = 1 \), \( \rho = 1 \), and \( \lambda = 1 \) and simulate the problem on the time interval $[0,1]$.
    The exact solution decides the source function and boundary conditions.
    
    We set the constant time step size \( \Delta t = \frac{1}{8}, \frac{1}{16}, \frac{1}{32}, \frac{1}{64} \) and fix 
    the uniform mesh diameter \( h = \frac{1}{128} \) to verify the convergence rate in time. 
    Meanwhile we adjust \( h = \frac{1}{16}, \frac{1}{32}, \frac{1}{64}, \frac{1}{128} \) and keep
    \( \Delta t = 1 \times 10^{-5} \) to confirm the convergence rate in space. 
    The results from \Cref{tab:$L^2$-errors and convergence rates in time,tab:$H^1$-errors and convergence rates in time,tab:$L^2$-errors and convergence rates in space,tab:$H^1$-errors and convergence rates in space} 
    is consistent with error estimates in Section \ref{sec:sec4}: 
	the approximate velocity $u$ has second order accuracy in both $L^2$ and $H^1$-norm while the approximate pressure converges at second order in $L^2$-norm.


	\begin{table}
		\centering
		\caption{$L^2$-errors and convergence rates in time}
		\begin{tabular}{@{}lllllllll@{}}
			\hline
			$1 / \Delta t$ & $\|u - u^h\|$ & Rate & $\|w - w^h\|$ & Rate & $\|\phi- \phi^h\|$ & Rate & $\|p - p^h \|$ & Rate \\
			\hline
			$4$   & 3.01E-1 & —         & 1.59E+1 & —         & 5.15E-1 & —         & 1.59E+2  & —         \\
			$8$   & 5.72E-2 & 2.3973 & 3.02E+0 & 2.3974  & 9.78E-2 & 2.3974 & 3.01E+1  & 2.4066   \\
			$16$ & 1.18E-2 & 2.2766 & 6.22E-1 & 2.2766  & 2.02E-2 & 2.2765  & 6.19E+0  & 2.2794    \\
			$32$ & 2.61E-3 & 2.1747 & 1.38E-1 & 2.1746  & 4.48E-3 & 2.1733  & 1.37E+0 & 2.1766   \\
			\hline
		\end{tabular}
		\label{tab:$L^2$-errors and convergence rates in time}
	\end{table}

	\begin{table}
		\centering
		\caption{$H^1$-errors and convergence rates in time}
		\begin{tabular}{@{}lllllllll@{}}
			\hline
			$1 / \Delta t$ & $\| u - u^h \|_{1}$ & Rate & $\|w - w^h\|_{1}$ & Rate & $\|\phi- \phi^h\|_{1}$ & Rate & $\|p - p^h\|_{1}$ & Rate \\
			\hline
			$4$   & 2.20E+0 & —        & 1.20E+2 & —         &6.09E+0 & —         &8.23E+2   & —          \\
			$8$   & 4.19E-1 & 2.3965 & 2.27E+1 & 2.3961 &1.17E+0 & 2.3800 &1.55E+2   &2.4044   \\
			$16$ & 8.73E-2 & 2.2615 & 4.79E+0 & 2.2466 &2.88E-1  &2.0237  &3.21E+1  & 2.2775    \\
			$32$ & 2.27E-2 & 1.9459 &1.40E+0  & 1.7773 &7.13E-2  & 2.0124 & 7.15E+0  & 2.1643    \\
			\hline
		\end{tabular}
		\label{tab:$H^1$-errors and convergence rates in time}
	\end{table}

    \begin{table}
        \caption{$L^2$-errors and convergence rates in space}
        \label{tab:L2_errors_convergence_space}
        \centering
        \begin{tabular}{@{}lllllllll@{}}
            \hline
            $1 / \Delta h$ & $\|u - u^h\|$ & Rate & $\|w - w^h\|$ & Rate & $\|\phi - \phi^h\|$ & Rate & $\|p - p^h\|$ & Rate \\
            \hline
            16   & 8.12E-4 & —      & 6.07E-2 & —      & 2.29E-3 & —      & 2.79E-1 & —      \\
            32   & 1.02E-4 & 2.9930 & 7.73E-3 & 2.9739 & 3.38E-4 & 2.7603 & 4.67E-2 & 2.5797 \\
            64   & 1.28E-5 & 2.9981 & 9.71E-4 & 2.9926 & 7.81E-5 & 2.1133 & 9.74E-3 & 2.2611 \\
            128  & 1.60E-6 & 2.9996 & 1.22E-4 & 2.9980 & 1.95E-5 & 2.0532 & 2.30E-3 & 2.0816 \\
            \hline
        \end{tabular}
        \label{tab:$L^2$-errors and convergence rates in space}
    \end{table}

    \begin{table}[h]
        \centering
        \caption{$H^1$-errors and convergence rates in space}
        \label{tab:H1-errors-space}
        \begin{tabular}{@{}lllllllll@{}}
            \hline
            $1 / \Delta h$ & $\|u - u^h\|_{1}$ & Rate & $\|w - w^h\|_{1}$ & Rate 
            & $\|\phi - \phi^h\|_{1}$ & Rate & $\|p - p^h\|_{1}$ & Rate \\
            \hline
            $16$  & 9.78E-2 & —      & 7.53E+0 & —      & 1.64E-1 & —      & 2.00E+1 & —      \\
            $32$  & 2.47E-2 & 1.9881 & 1.90E+0 & 1.9883 & 7.71E-2 & 1.0863 & 1.06E+1 & 0.9215 \\
            $64$  & 6.18E-3 & 1.9969 & 4.75E-1 & 1.9970 & 3.83E-2 & 1.0093 & 4.96E+0 & 1.0887 \\
            $128$ & 1.55E-3 & 1.9992 & 1.19E-1 & 1.9992 & 1.95E-2 & 0.9738 & 2.44E+0 & 1.0250 \\
            \hline
        \end{tabular}
        \label{tab:$H^1$-errors and convergence rates in space}
    \end{table}

    \subsection{Two dimensional self-organization of active fluid}
    \label{subsec:self-organization}
    To investigate long-time stability of the fully discrete DLN scheme (Scheme \ref{fully discrete formulations scheme second} or Scheme \ref{fully discrete formulations scheme}), 
    we refer to the numerical experiment about the spatial self-organization of the bacterial active fluid on the unit square domain \cite{LSMW21_Nature}.
    After a very short period of self-adjustment, the phase space trajectory of the bacterial active fluid should rotate clockwise or counterclockwise at a constant angle to form a unidirectional vortex current and maintain the steady state thereafter. 
    The bacterial active fluid on the domain $D = [0,1]^2$ satisfies the no-slip boundary condition \( u|_{\partial D} = w|_{\partial D} = 0 \), and the following random initial condition 
    \[
    u_0(x,y) =  \big( \text{rand}(x,y), \text{rand}(x,y) \big), \qquad (x,y) \in D,
    \]
    where 'rand'  is a uniform random generator in \([-1,1]\). 
    We set the parameters of the model to be
    $\mu = 0.045,  \nu = 0.003, \beta = 0.5, \alpha = -0.81, \gamma = \mu^3$ and the external force $f = 0$ 
    We simulate the test over the time interval \( [0,1] \) with the uniform time step size of \( \Delta t = 1/100 \).
	
	Figure~\ref{fig:velocity_glyph_comparison} illustrates the evolution of the vector field of velocity and the velocity field over time. 
	Initially, the velocity field exhibits significant disorder, and then small vortices begin to form by $t = 0.03$, though the overall structure remains chaotic. 
	As the time $t$ reaches \( 0.10 \), distinct vortical patterns emerge, leading to a more structured flow. 
	Eventually, the system transitions into a fully ordered polar state after \( t = 1.00 \), which implies that the fully discrete DLN scheme is long-time stable.

    \begin{figure}[htbp] 
        \centering
        \begin{minipage}[t]{0.3\linewidth}
            \centering
            \includegraphics[width=3.5cm]{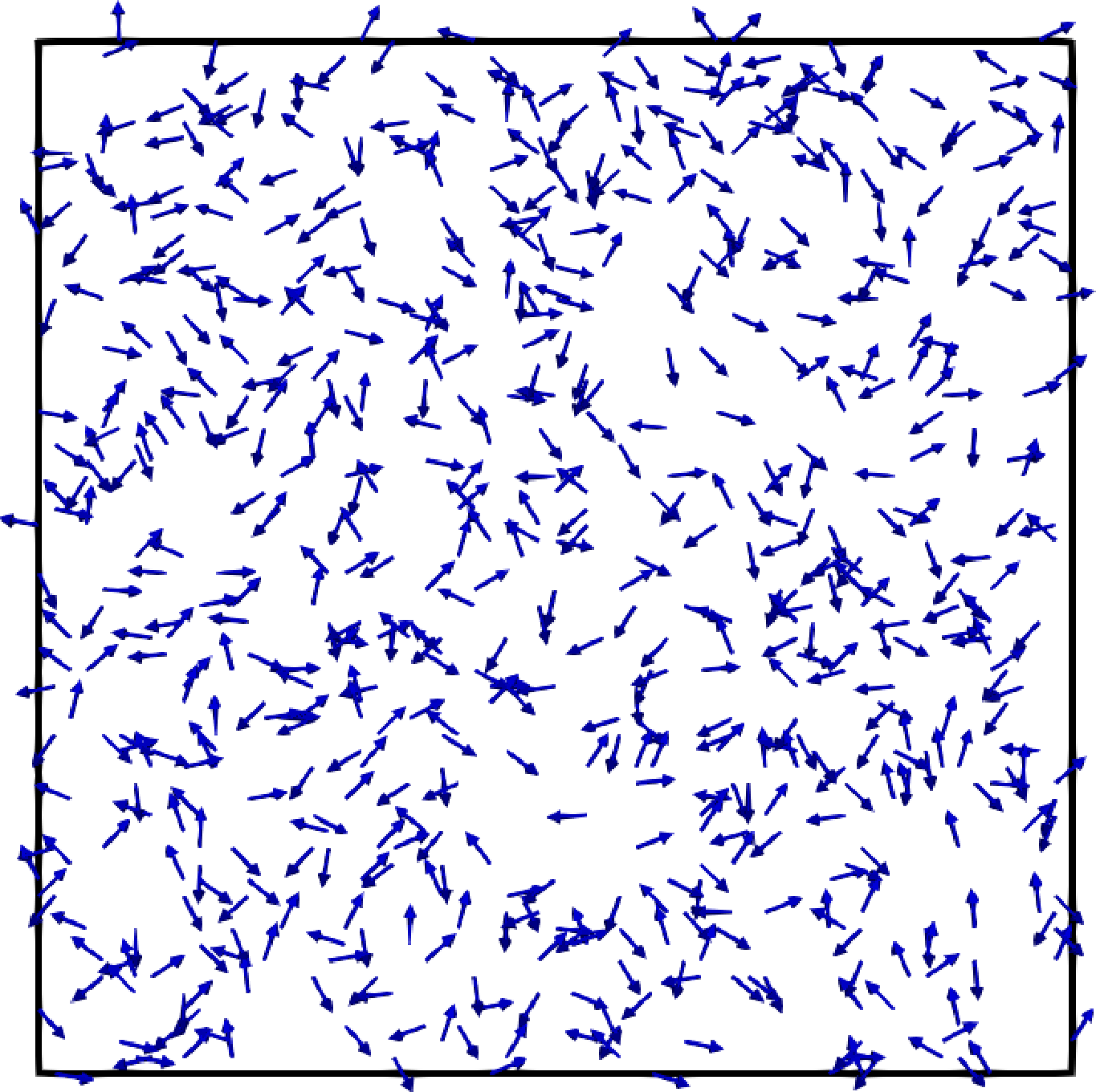} 
            \caption*{$t=0$}
        \end{minipage}
        \begin{minipage}[t]{0.3\linewidth}
            \centering
            \includegraphics[width=3.5cm]{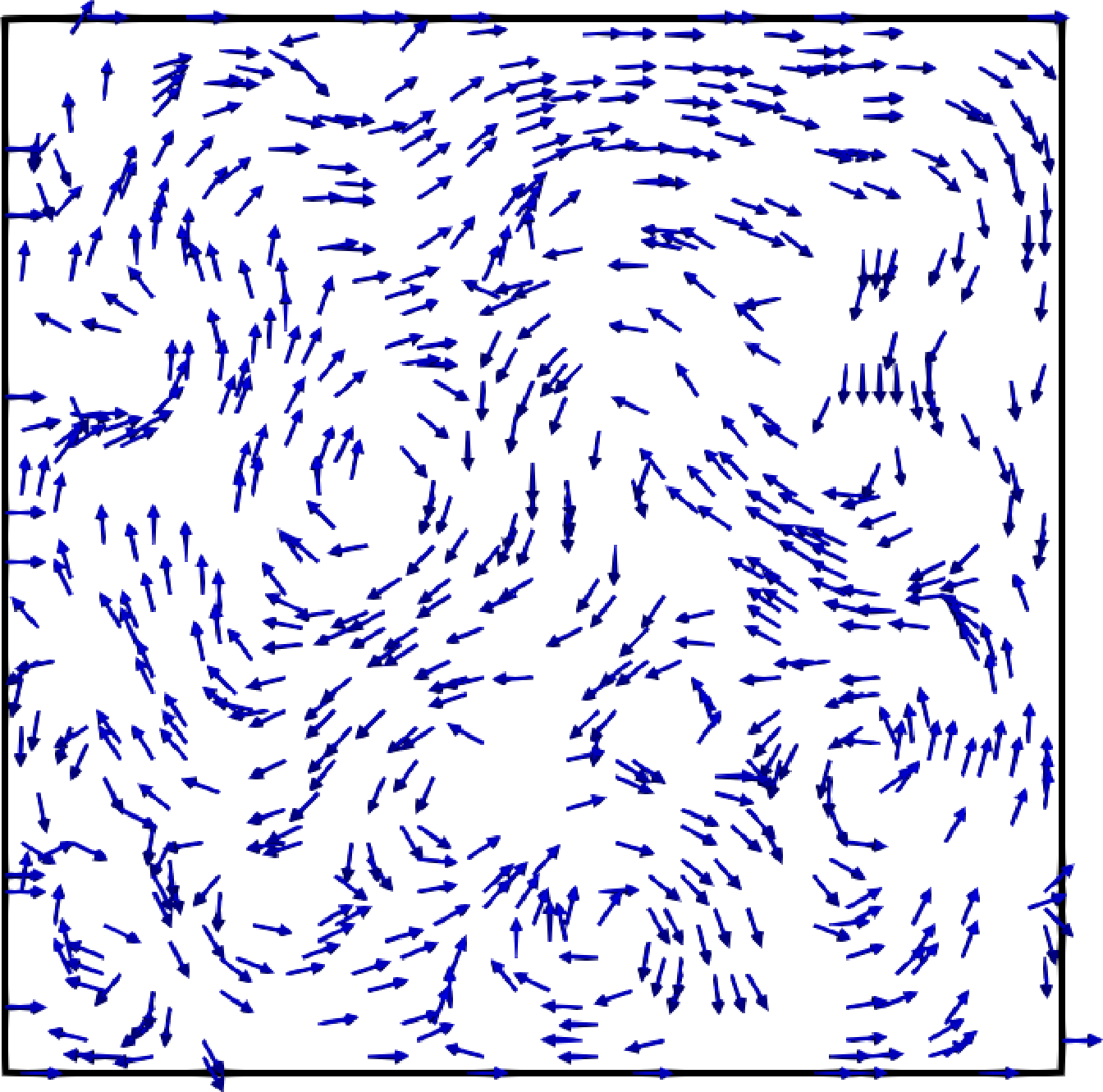} 
            \caption*{$t=0.03$}
        \end{minipage}
        \begin{minipage}[t]{0.3\linewidth}
            \centering
            \includegraphics[width=3.5cm]{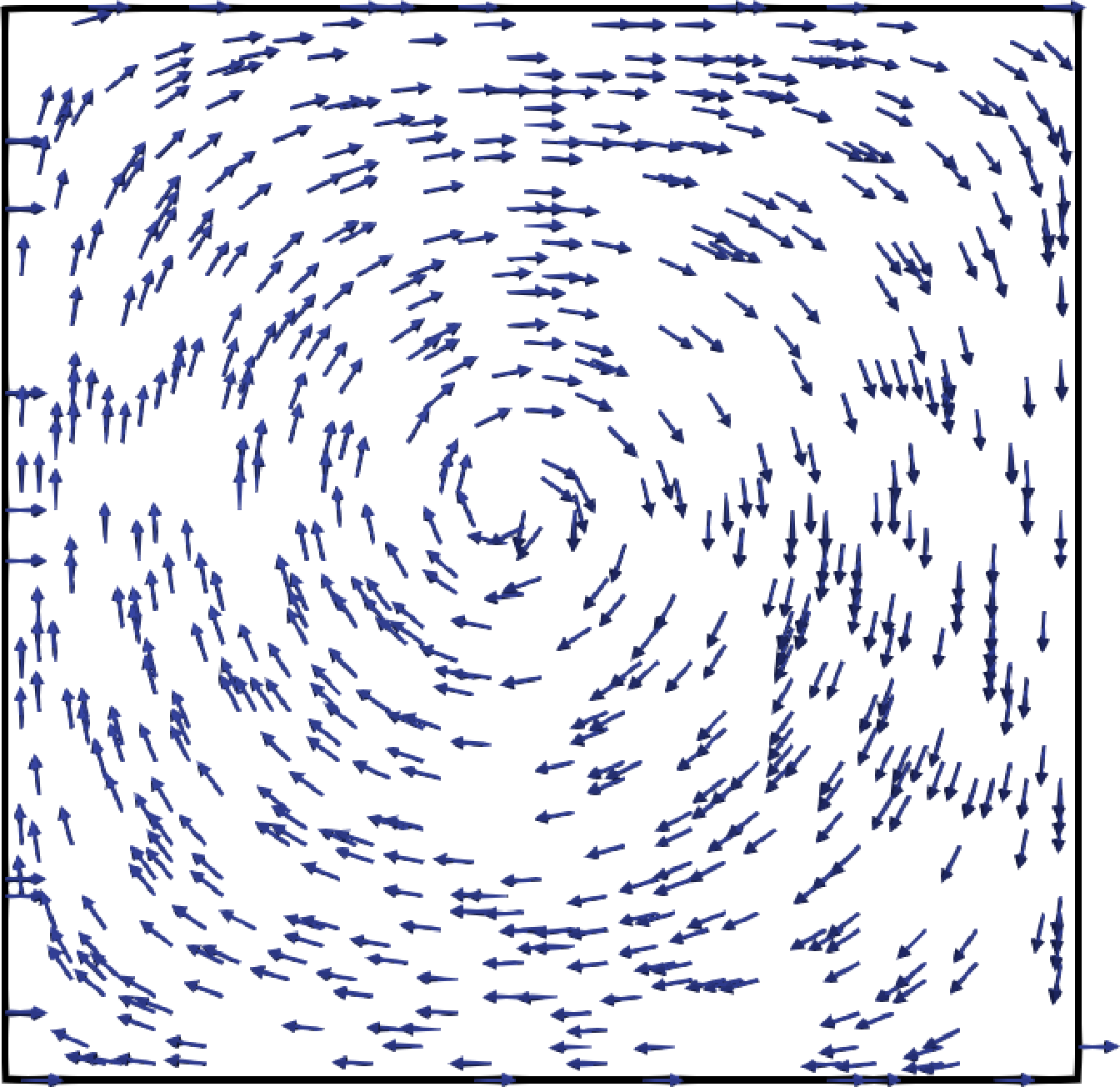} 
            \caption*{$t=1.00$}
        \end{minipage} 
        \vspace{-0.5cm}
        \caption{Time evolution of the vector field of velocity (top row) and the velocity field (bottom row) at selected time instances.}
        \label{fig:velocity_glyph_comparison}
    \end{figure}

    \subsection{Time adaptive test}
    To relieve the conflict between accuracy and computational cost as well as utilize fine properties of the DLN method under non-uniform time grids, we design a time-adaptive approach for the fully discrete DLN scheme (Scheme \ref{fully discrete formulations scheme second} or Scheme \ref{fully discrete formulations scheme}) based on the minimum dissipation criterion proposed by Capuano, Sanderse, De Angelis, and Coppola \cite{capuano2017minimum}. 
    At each time step, we compute the numerical dissipation (ND) rate for both velocity $u$ and auxiliary variable $w$,
    the viscosity-induced dissipation (VD) for $u$, generic stability-induced dissipation (SD) for $w$
    \begin{align*}
    &\text{ND: }\ \ 
    \epsilon_{ND}^{u} = \frac{1}{\widehat{k}_n} \| u_{n,\alpha}^{h} \|^2, 
    \quad
    \epsilon_{ND}^{w} = \frac{1}{\widehat{k}_n} \| w_{n,\alpha}^{h} \|^2, 
    \\
    &\text{VD: }\ \ 
    \epsilon_{VD}^{u} = \mu \| \nabla u_{n,\beta}^h \|^2,
    \quad
    \text{SD: }\ \ 
    \epsilon_{VD}^{w} = \gamma \| \nabla w_{n,\beta}^h \|^2,
    \end{align*}
    and the ratios of ND over VD and SD: $\chi_u = \epsilon_{ND}^{u} / \epsilon_{VD}^{u}$, $\chi_w = \epsilon_{ND}^{w} / \epsilon_{VD}^{w}$.
    Then we adjust the next time step $k_{n+1}$ by 
    \begin{equation}
    \label{time-controller}
    \begin{split}
    &k_{n+1} = 
    \begin{cases} 
    \min \{2 k_n, k_{\max} \}, & \text{if}~  \max \{|\chi_u|,|\chi_w|\} \leq \delta, \\
    \max \{\frac{1}{2}k_n, k_{\min} \}, & \text{if}~ \max \{|\chi_u|,|\chi_w|\} > \delta,
    \end{cases}
    \end{split}
    \end{equation}
	for the required tolerance $\delta >0$.
	We observe from \eqref{time-controller} that the strategy allows a larger time step ($k_{n+1} = 2k_{n}$) for the next operation if the ratios are below the required tolerance; otherwise, decrease the next time step by half. 
	Meanwhile, we set a maximum time step $k_{\max}$ for accuracy and a minimum time step $k_{\min}$ for efficiency. 
    
    We evaluate the performance of the proposed time-adaptive strategy via the former experiment 
    in Subsection \ref{subsec:self-organization}. 
    We set $k_{\max} = 0.01$, $k_{\min} = 1.\rm{e}-5$, $\delta = 2$, $k_{0} = k_{\min}$, carry out the experiment with different level of Reynolds number $\rm{Re} = 1/\mu$: $300, 500, 3000, 5000, 10000, 50000$, and make other parameters and conditions unchanged. 
    We also compare this approach against the corresponding constant time-stepping scheme with $10000$ time steps for the effectiveness of time adaptivity. 
	Both approaches achieve very similar results of  the evolution of the vector field of velocity and the velocity field over time interval $[0,1]$. 
	However, \Cref{Comparison of Different Reynolds Numbers with Adaptive Time Steps and Constant Time Steps} shows that the constant time-stepping scheme costs many more time steps under all levels of Reynolds number selected, which emphasizes the superiority of the time-adaptive approach.
    \begin{table}
        \centering
        \caption{The constant time-stepping scheme costs many more time steps under all levels of Reynolds number selected, which emphasizes the superiority of the time-adaptive approach.}
        \begin{tabular}{@{}lrrrrrr@{}}
            \hline
            Re$^a$  & 300 & 500 & 3000 & 5000 & 10,000& 50,000 \\
            \hline
            Adaptive$^b$   &  166 &  146 & 457& 406  & 2035&2979\\
            Constant$^c$  &10000& 10000 & 10000 & 10000  & 10000& 10000  \\
            \hline
        \end{tabular}
        \label{Comparison of Different Reynolds Numbers with Adaptive Time Steps and Constant Time Steps}
        
        \vspace{1ex}
        {\footnotesize
            $^a$ Reynolds number; \\
            $^b$ Number of computational steps with adaptive time stepping; \\
            $^c$ Number of computational steps with fixed time steps.
        }
    \end{table}

    \section{Conclusion}
    \label{sec:sec6}
    In this paper, we have developed an efficient spatio-temporal discretization scheme for an equivalent second-order reformulation of the active fluid system.
    The variable time-stepping DLN method, which is second-order accurate and unconditionally nonlinearly stable, is employed as the time integrator.
    For spatial discretization, we introduce an additional auxiliary variable, thus construct a divergence-free preserving and easily-implemented mixed finite element method.
    With the help of appropriate regularity assumptions and mild time-diameter restrictions, we have rigorously proved that the fully discrete DLN scheme is long-time stable in kinetic energy and established error estimates for velocity in $L^2$ and $H^1$-norm and pressure in $L^2$-norm.
    Furthermore, a time-adaptive strategy is designed to maintain robustness of the scheme and improve computational efficiency.
    Several numerical experiments validate our theoretical findings, demonstrating that the fully discrete DLN scheme, along with the time-adaptive approach, provides an efficient framework for solving active fluid dynamics and other more complex systems.

    \bibliographystyle{abbrv}

    \bibliography{bibliography}

\end{document}